\documentclass{amsart}
\usepackage{amsmath,amssymb,amscd}
\usepackage[colorlinks=black,linkcolor=black]{hyperref}
\usepackage{cite}
\usepackage{geometry}
\newtheorem{theorem}{Theorem}[section]
\newtheorem{lemma}[theorem]{Lemma}
\newtheorem{proposition}[theorem]{Proposition}
\newtheorem{corollary}[theorem]{Corollary}

\newtheorem{remark}[theorem]{Remark}
\numberwithin{equation}{section}
\allowdisplaybreaks[4]
\DeclareUnicodeCharacter{2212}{-}

\begin{document}

\title[Spectral Invariants of the Magnetic Dirichlet-to-Neumann Map]{Spectral Invariants of the Magnetic Dirichlet-to-Neumann Map on Riemannian Manifolds}

\author{Genqian Liu}
\address{School of Mathematics and Statistics, Beijing Institute of Technology, Beijing 100081, China}
\email{liugqz@bit.edu.cn}

\author{Xiaoming Tan}
\address{School of Mathematics and Statistics, Beijing Institute of Technology, Beijing 100081, China}
\email{xtan@bit.edu.cn}


\begin{abstract}
    This paper is devoted to investigate the heat trace asymptotic expansion corresponding to the magnetic Steklov eigenvalue problem on Riemannian manifolds with boundary. We establish an effective procedure, by which we can calculate all the coefficients $a_0$, $a_1$, $\dots$, $a_{n-1}$ of the heat trace asymptotic expansion. In particular, we explicitly give the expressions for the first four coefficients. These coefficients are spectral invariants which provide precise information concerning the volume and curvatures of the boundary of the manifold and some physical quantities by the magnetic Steklov eigenvalues.
\end{abstract}

\maketitle

\section{Introduction}

Let $(\Omega,g)$ be an $n$-dimensional, smooth, compact Riemannian manifold with smooth boundary $\partial \Omega$. The magnetic potential $A \in \mathfrak{X}(\Omega)$ is a vector field with smooth complex-valued coefficients, and the electric potential $q \in C^{\infty}(\Omega)$ is a smooth complex-valued function on $\Omega$. We denote by $\operatorname{grad}$, $\operatorname{div}$ and $\langle \cdot,\cdot \rangle$ the gradient operator, the divergence operator and the inner product on $\Omega$ with respect to the Riemannian metric $g$, respectively. The magnetic Schr\"{o}dinger operator $\mathcal{L}_{g}$ (also called Schr\"{o}dinger operator with magnetic fields) on the Riemannian manifold is defined by
\begin{equation}\label{1.2}
    \mathcal{L}_{g} u = - \Delta_{g} u - 2i \langle A,\operatorname{grad}u \rangle + Vu
\end{equation}
for $u \in C^{\infty}(\Omega)$, where $i = \sqrt{-1}$, $\Delta_{g}$ is the Laplace-Beltrami operator, and
\begin{align}\label{1.3}
    V = |A|^2 - i \operatorname{div} A + q - k^2, \quad k \in \mathbb{R}.
\end{align}
Particularly, the magnetic Schr\"{o}dinger operator $\mathcal{L}_{g}$ reduces to the Schr\"{o}dinger operator $-\Delta_{g} + V$ when $A = 0$, and it further reduces to the Laplace-Beltrami operator $-\Delta_{g}$ when $A = V = 0$.

The magnetic Schr\"{o}dinger operator appears naturally in some mathematical models related to quantum physics because the magnetic Schr\"{o}dinger equation $\mathcal{L}_{g} u = 0$ describes the behavior of a quantum non-relativistic particle under the influence of the external electric and magnetic fields (see \cite{AHB78.1}). In quantum physics, the history of the magnetic Schr\"{o}dinger operator can be traced back to the last century. The rigorous mathematical theory of this operator has started with studying the classical Schr\"{o}dinger operator $-\Delta + q$ and its multi-particle analogues. Here the Laplace operator describes the kinetic energy of the particle and $q$ is the electric potential. Although both of the kinetic and potential energy operators are quite simple to study separately, their sum exhibits a rich variety of complex phenomena which differ from their classical counterparts in many aspects. The mathematical theory of the classical Schr\"{o}dinger operator is the most developed and most extensive in mathematical physics (see \cite{Simon00}).

In order to study the behavior of a particle under the influence of the external magnetic fields, magnetic potentials are included in the theory of classical Schr\"{o}dinger operator, but spins are neglected. In this situation, the kinetic energy operator is modified from $-\Delta$ to $(-i \nabla + A)^2$ by the minimal substitution rule. Here $-i \nabla$ is the momentum operator, and $A$ is the magnetic vector potential that generates the magnetic field. The operator $(-i \nabla+A)^2 + q$ is called the classical magnetic Schr\"{o}dinger operator. It is important to note that the theory of the classical magnetic Schr\"{o}dinger operator is more complicated than that of $-\Delta + q$. Actually, the kinetic energy part contains noncommuting operators, in other words, $\big[(-i \nabla + A)_k,(-i \nabla + A)_l\big] \neq 0$ for $k \neq l$, the commutator is the magnetic field (up to a factor $i$). Notice that the magnetic spectrum is characteristically different from that of the free Laplacian (see \cite{Erdos07}). The magnetic Schr\"{o}dinger operator has a wide range of applications in theoretical physics, such as the Ginzburg-Landau theory of superconductors (see \cite{BCP12}, \cite{EPV08}, \cite{FourHelf10}), the theory of Bose-Einstein condensates (see \cite{KapKev05}), and the study of edge states in quantum mechanics (see \cite{Frank07}, \cite{RagHal08}).

\addvspace{2mm}

Let $\bar{v}$ be the complex conjugation of $v \in C^{\infty}(\Omega)$, and $\nu$ be the outward unit normal vector to $\partial \Omega$. By integrating by parts, one has
\begin{align*}
    \int_{\Omega} (\mathcal{L}_{g} u)\bar{v} \,dV
    & = \int_{\Omega} \langle \operatorname{grad} u + iuA,\ \operatorname{grad} \bar{v} - i\bar{v}A \rangle \,dV + \int_{\Omega} (q - k^2) u\bar{v} \,dV \\
    & \quad - \int_{\partial \Omega} \langle \operatorname{grad} u + iuA,\ \nu \rangle \bar{v} \,dS \quad \text{for}\ u,v \in C^{\infty}(\Omega), 
\end{align*}
where $dV$ is the Riemannian volume element on $\Omega$, and $dS$ is the Riemannian volume element on $\partial \Omega$ (with respect to the orientation induced from $\Omega$). Actually, the magnetic Schr\"{o}dinger operator is a self-adjoint operator (the self-adjointness of this operator is just a requirement for the operator to describe an energy observable in quantum mechanics, or see \cite{IkebeKato62}, \cite{Kato72}, \cite{Knowles77}, \cite{Sche75}, \cite{Simon00}, \cite{Simon73}).

The Dirichlet-to-Neumann map $\mathcal{M}: H^{\frac{1}{2}}(\partial \Omega) \to H^{-\frac{1}{2}}(\partial \Omega)$ associated with the magnetic Schr\"{o}dinger operator $\mathcal{L}_{g}$ is defined by
\begin{equation}\label{1.5}
    \mathcal{M}(f)
    := \langle \operatorname{grad} u + iuA,\nu \rangle \big|_{\partial \Omega}
    = (\partial_\nu u + i\langle A,\nu \rangle u)\big|_{\partial \Omega},
\end{equation}
where $\langle \operatorname{grad} u + iuA,\nu \rangle |_{\partial \Omega}$ is called the magnetic normal derivative on the boundary (magnetic Neumann boundary condition), and $u$ solves the corresponding Dirichlet problem:
\begin{equation}\label{1.4}
    \begin{cases}
        \mathcal{L}_{g} u = 0 \quad & \text{in}\ \Omega, \\
        u = f & \text{on}\ \partial \Omega.
    \end{cases}
\end{equation}
We always assume that $0$ is not a Dirichlet eigenvalue for the magnetic Schr\"{o}dinger operator $\mathcal{L}_{g}$, by properly selecting a real number $k$, so that the Dirichlet problem \eqref{1.4} has a unique solution. The map $\mathcal{M}$ is called the magnetic Dirichlet-to-Neumann map. Obviously, $\mathcal{M}$ is a self-adjoint, first order pseudodifferential operator (see \cite{NSU95}), and it reduces to the classical Dirichlet-to-Neumann map when $A = V = 0$.

\addvspace{2mm}

In the present paper, we consider the following magnetic Steklov eigenvalue problem:
\begin{equation}\label{1.4.1}
    \begin{cases}
        \mathcal{L}_{g} u = 0 \quad & \text{in}\ \Omega, \\
        \mathcal{M}(u) = \lambda u & \text{on}\ \partial \Omega.
    \end{cases}
\end{equation}
The spectrum of the magnetic Steklov eigenvalue problem \eqref{1.4.1} consists of a discrete sequence
\begin{align*}
    0 \leqslant \lambda_{1} \leqslant \lambda_{2} \leqslant \cdots \leqslant \lambda_{k} \leqslant \cdots \to \infty
\end{align*}
with each eigenvalue repeated according to its multiplicity. The corresponding eigenfunctions $\{u_k\}_{k \geqslant 1}$ form an orthogonal basis in $L^{2}(\partial \Omega)$. Note that the ratio of the magnetic normal derivative to the wave function is just the magnetic Steklov eigenvalue.

This problem originates from inverse spectral problems. One hopes to recover the geometry of a manifold from the set of the known data (the spectrum of a differential or pseudodifferential operator). Notice that the spectrum $\{\lambda_k\}_{k \geqslant 1}$ of the magnetic Steklov eigenvalue problem \eqref{1.4.1} is just the spectrum of the magnetic Dirichlet-to-Neumann map $\mathcal{M}$, and these eigenvalues are physical quantities which can be measured experimentally. Apparently, the knowledge provided by the map $\mathcal{M}$ (i.e., the set of the Cauchy data $\{(u|_{\partial \Omega},(\partial_\nu u + i\langle A,\nu \rangle u)\big|_{\partial \Omega})\}$) is equivalent to the information given by all of the magnetic Steklov eigenvalues and the corresponding eigenfunctions.

Here we briefly recall the classical Steklov eigenvalue problem associated with the Laplace-Beltrami operator as follows:
\begin{equation*}
    \begin{cases}
        \Delta_g u = 0 \quad & \text{in}\ \Omega, \\
        \frac{\partial u}{\partial \nu} = \tau u & \text{on}\ \partial \Omega.
    \end{cases}
\end{equation*}
The classical Dirichlet-to-Neumann map $\Lambda$ (also known as the voltage-to-current map) is defined by $\Lambda(f) := \frac{\partial u}{\partial \nu}\big|_{\partial \Omega}$, where $u$ solves the Dirichlet problem: $\Delta_g u = 0$ in $\Omega$, $u = f$ on $\partial \Omega$. This problem is an eigenvalue problem with the spectral parameter in the boundary condition. The study of the spectrum of the map $\Lambda$ was initiated by Steklov \cite{Steklov02} in 1902. Notice that the Steklov spectrum coincides with that of the classical Dirichlet-to-Neumann map. Eigenvalues and eigenfunctions of the map $\Lambda$ have a number of applications in physics, such as fluid mechanics, heat transmission and vibration problems (see \cite{FoxKuttler83}, \cite{KopaKrein01}). The Steklov spectrum also plays a fundamental role in the mathematical analysis of photonic crystals (see \cite{Kuch01}). The classical Steklov eigenvalue problem has attracted a lot of attention in recent years.

It follows from a famous result of \cite{Sandgren55} that the eigenvalue counting function $N(\tau)$ of the classical Dirichlet-to-Neumann map satisfies the Sandgren’s asymptotic formula
\begin{equation}\label{1.4.2}
    N(\tau)
    = \#(k|\tau_{k} \leqslant \tau)
    = \frac{\operatorname{vol}(B^{n-1})}{(2\pi)^{n-1}} \operatorname{vol}(\partial \Omega) \tau^{n-1} + o(\tau^{n-1}) \quad \text{as}\ \tau \to +\infty,
\end{equation}
or equivalently,
\begin{equation*}
    \operatorname{Tr} e^{-t \Lambda} 
    = \sum_{k=1}^{\infty} e^{-t \tau_{k}} \sim \frac{\Gamma(n) \operatorname{vol}(B^{n-1})}{(2\pi)^{n-1}t^{n-1}} \operatorname{vol}(\partial \Omega) \quad \text{as}\ t \to 0^+, 
\end{equation*}
where $B^{n-1}$ is the $(n-1)$-dimensional unit ball. This shows that one can ``hear'' the volume $\operatorname{vol}(\partial \Omega)$ of the boundary from the first term of the asymptotic expansion above. Note that a sharp form of \eqref{1.4.2} was given by Liu \cite{Liu11}.

Generally, for the eigenvalues of the classical Dirichlet-to-Neumann map $\Lambda$ associated with the Laplace-Beltrami operator $\Delta_{g}$, the trace $\operatorname{Tr} e^{-t \Lambda}$ of the associated heat kernel admits an asymptotic expansion
\begin{equation*}
    \operatorname{Tr} e^{-t \Lambda} 
    = \sum_{k=1}^{\infty} e^{-t \tau_{k}} \sim \sum_{k=0}^{n-1} \tilde{a}_{k} t^{-n+k+1} + o(1)\quad \text{as}\ t \to 0^+,
\end{equation*}
where the coefficients $\tilde{a}_k = \int_{\partial \Omega} \tilde{a}_k(x) \,dS$ are spectral invariants of the classical Dirichlet-to-Neumann map. The first four coefficients $\tilde{a}_k(x)\ (k=0,1,2,3)$ have been obtained by Liu \cite{Liu15}. The first four relative heat invariants have been calculated in \cite{WangWang19} for the Schr\"{o}dinger operator $-\Delta_{g} + V$. For other geometric invariants, we refer the reader to Liu's results \cite{Liu19}, \cite{Liu_NL21} for Navier-Lam\'{e} operator, \cite{Liu_S21} for Stokes operator, and \cite{Liu16}, \cite{Liu11} for biharmonic operator. See also \cite{BranGilk90}, \cite{Gilkey04}, \cite{Gilkey95}, \cite{Gilkey75}, \cite{GilkeyGrubb98} for more results of this topic.

The classical Dirichlet-to-Neumann map is also connected with the famous Calder\'{o}n problem (see \cite{Cald80}), which is to determine the conductivity in the interior of a conductor by boundary measurements from the knowledge of the map $\Lambda$. After the seminal paper \cite{SylvUhlm87}, plenty of similar problems for elliptic equations have been intensively investigated (see \cite{Uhlmann14} for a survey). Such as the inverse boundary value problem for the magnetic Schrödinger operator, which is to determine $q$ and $A$ (up to a gauge transformation) from the knowledge of the magnetic Dirichlet-to-Neumann map in Euclidean spaces or Riemannian manifolds. Inverse problems for magnetic Schr\"{o}dinger operators have been studied by numerous authors in different settings (see \cite{BellChou10}, \cite{DKSU09}, \cite{DKSU07}, \cite{Eskin01}, \cite{KrupUhlm18}, \cite{KrupUhlm14}, \cite{NSU95}, \cite{PSU10}, \cite{Sun93}).

\addvspace{2mm}

It is natural to raise the following interesting open problem: what geometric information about the manifold can be obtained explicitly by providing all of the magnetic Steklov eigenvalues? Once this problem is solved, one can recover the geometry of a manifold from the spectrum of the magnetic Dirichlet-to-Neumann map. This problem is analogous to the well-known Kac’s problem \cite{Kac66} (i.e., is it possible to ``hear'' the shape of a domain just by ``hearing'' all of the eigenvalues of the Dirichlet Laplacian? See also \cite{Loren47}, \cite{Protter87}, \cite{Weyl12}). In this paper, we give an affirmative answer for this open problem mentioned above.

To obtain more geometric information about the manifold, we study the magnetic Steklov eigenvalue problem \eqref{1.4.1} and the asymptotic expansion of the trace $\operatorname{Tr} e^{-t \mathcal{M}}$ of the heat kernel associated with the magnetic Dirichlet-to-Neumann map $\mathcal{M}$ (the so-called ``heat kernel method''). The coefficients of the asymptotic expansion are spectral invariants which are metric invariants of the boundary $\partial \Omega$ of the manifold, and they contain a lot of geometric and topological information about the manifold (see \cite{Edward91}, \cite{Gilkey75}, \cite{Kac66}). Spectral invariants are of great importance in spectral geometry, they are also connected with many physical concepts (see \cite{ANPS09}, \cite{Full94}) and have increasingly extensive applications in physics since they describe corresponding physical phenomena. The more spectral invariants we get, the more geometric and topological information about the manifold we know. However, computations of spectral invariants are challenging problems in spectral geometry (see \cite{Gilkey75}, \cite{Grubb86}, \cite{Liu11}, \cite{SafarovVass97}). Let us point out that the complexity of the magnetic Schr\"{o}dinger operator directly lead to the calculations of some quantities (for example, spectral invariants) become much more difficult than that of Laplace-Beltrami operator. To the best of our knowledge, it has not previously been systematically applied in the context of the magnetic Steklov problem. By the theory of pseudodifferential operators and symbol calculus, we can calculate all the coefficients $a_0$, $a_1$, $\dots$, $a_{n-1}$ of the magnetic Dirichlet-to-Neumann map, and explicitly give the expressions for the first four coefficients.

\addvspace{2mm}

The main result of this paper is the following theorem.
\begin{theorem}\label{thm1.1}
    Suppose that $(\Omega,g)$ is an $n$-dimensional, smooth, compact Riemannian manifold with smooth boundary $\partial \Omega$. Let $\{\lambda_k\}_{k \geqslant 1}$ be the eigenvalues of the magnetic Dirichlet-to-Neumann map $\mathcal{M}$. Then, the trace of the heat kernel associated with the map $\mathcal{M}$ admits an asymptotic expansion
    \begin{equation*}
        \operatorname{Tr} e^{-t \mathcal{M}} 
        = \sum_{k=1}^{\infty} e^{-t \lambda_{k}} \sim \sum_{k=0}^{n-1} a_{k} t^{-n+k+1} + o(1)\quad \text{as}\ t \to 0^+, 
    \end{equation*}
    where the coefficients $a_k = \int_{\partial \Omega} a_k(x) \,dS$ are spectral invariants of the magnetic Dirichlet-to-Neumann map, which can be explicitly calculated for $0 \leqslant k \leqslant n-1$.
    
    In particular, the first four coefficients are given by
    \begin{align*}
        a_0(x) & = \frac{\Gamma(n-1)\operatorname{vol}(\mathbb{S}^{n-2})}{(2\pi)^{n-1}}, \quad n \geqslant 1; \\[1mm]
        a_1(x) & = \frac{(n-2)\Gamma(n-1)\operatorname{vol}(\mathbb{S}^{n-2})}{2(2\pi)^{n-1}(n-1)}H, \quad n \geqslant 2; \\[1mm]
        a_2(x) & = \frac{\Gamma(n-2)\operatorname{vol}(\mathbb{S}^{n-2})}{24(2\pi)^{n-1}(n^2-1)}
        \biggl[
            3(n^3-4n^2+n+8)H^2 + 3n(n-3) \sum_\alpha \kappa_\alpha^2 \\
            & \quad + 3(n+1)(n-2)\tilde{R} - (n+1)(n-4)R - 12(n^2-1)q + 12(n^2-1)k^2
        \biggr], \quad n \geqslant 3,
    \end{align*}
    and
    \begin{align*}
        a_3(x) & = \frac{\Gamma(n-3)\operatorname{vol}(\mathbb{S}^{n-2})}{48(2\pi)^{n-1}(n+3)(n^2-1)}
        \biggl(
            \alpha_{1}H^3 + \alpha_{2}H\tilde{R} + \alpha_{3}HR + \alpha_{4}H\sum_\alpha \kappa_\alpha^2  + \alpha_{5}\sum_\alpha \kappa_\alpha^3 \\
            & \quad + \alpha_{6}\sum_\alpha \kappa_\alpha \tilde{R}_{\alpha\alpha} + \alpha_{7}\sum_\alpha \kappa_\alpha R_{\alpha\alpha} + \alpha_{8} \nabla_{\frac{\partial }{\partial x_n}}\tilde{R}_{nn} + \alpha_{9}\frac{\partial q}{\partial x_{n}} + \alpha_{10}Hq + \alpha_{11}k^2H
        \biggr), \quad n \geqslant 4.
    \end{align*}
    Here $\kappa_\alpha$ are the principal curvatures of the boundary $\partial \Omega$, $H = \sum_\alpha \kappa_\alpha$ is the mean curvature of $\partial \Omega$, $R_{\alpha\alpha}$ are the components of the Ricci tensor of $\partial \Omega$, $R$ is the scalar curvature of $\partial \Omega$, $\tilde{R}_{\alpha\alpha}$ are the components of the Ricci tensor of $\Omega$, $\tilde{R}$ is the scalar curvature of $\Omega$, $\nabla_{\frac{\partial }{\partial x_n}}\tilde{R}_{nn}$ is the covariant derivative of $\tilde{R}_{nn}$ with respect to $\frac{\partial }{\partial x_n}$, and $\alpha_{k}\ (1 \leqslant k \leqslant 11)$ are as follows:
    \begin{align*}
        & \alpha_{1} = n^5-5n^4-10n^3+52n^2+2n-114,\\
        & \alpha_{2} = 3(n+3)(n^3-6n^2+2n+14),\\
        & \alpha_{3} = -(n+3)(3n^3-20n^2+12n+42),\\
        & \alpha_{4} = 3(n^4-n^3-12n^2+22n+6),\\
        & \alpha_{5} = -8(n-2)(n-3),\\
        & \alpha_{6} = 12(n+3)(n^2-3n+1),\\
        & \alpha_{7} = -4n(n+3)(3n-8),\\
        & \alpha_{8} = 6(n+1)(n+3)(n-2),\\
        & \alpha_{9} = -12(n+3)(n^2-1),\\
        & \alpha_{10} = -12(n+3)(n-4)(n^2-1),\\
        & \alpha_{11} = 12(n+3)(n-4)(n^2-1).
    \end{align*}
\end{theorem}

\addvspace{5mm}

\begin{remark}
    (\romannumeral 1) It follows from Theorem \ref{thm1.1} that the volume $\operatorname{vol}(\partial \Omega)$ of the boundary, the total mean curvature $\int_{\partial \Omega}H \,dS\ (n \geqslant 3)$ of the boundary, some other total curvatures as well as some physical quantities are all spectral invariants of the magnetic Dirichlet-to-Neumann map, and these invariants can also be obtained by the magnetic Steklov eigenvalues. This shows that one can ``hear'' all the coefficients $a_0$, $a_1$, $\dots$, $a_{n-1}$ (i.e., spectral invariants) by ``hearing'' all of the eigenvalues of the magnetic Dirichlet-to-Neumann map.

    \addvspace{2mm}

    (\romannumeral 2) By applying the Tauberian theorem (see, Theorem 15.3 of p.~30 of \cite{Kore04} or p.~107 of \cite{Taylor96.2}) for $a_0(x)$, we immediately get the Sandgren’s asymptotic formula \eqref{1.4.2}.

    \addvspace{2mm}

    (\romannumeral 3) When $n = 2$, we immediately get $a_1(x) = 0$, which generalize the result of J.~Edward and S.~Wu \cite{Edward91} for the classical Dirichlet-to-Neumann map associated with the Laplace operator for any simply connected bounded domain $\Omega \subset \mathbb{R}^2$.

    \addvspace{2mm}

    (\romannumeral 4) Theorem \ref{thm1.1} is a generalization of both \cite{Liu15} for Laplace-Beltrami operator $\Delta_{g}$ and \cite{WangWang19} for Schr\"{o}dinger operator $-\Delta_{g} + V$.
\end{remark}

\addvspace{5mm}

The expressions for $a_2(x)$ and $a_3(x)$ in Theorem \ref{thm1.1} can be simplified when the manifold $\Omega$ has constant sectional curvature. Thus, we have the following
\begin{corollary}\label{cor1.3}
    Assuming the manifold $\Omega$ has constant sectional curvature $K$, we then obtain
    \begin{align*}
        a_2(x) & = \frac{\Gamma(n-2)\operatorname{vol}(\mathbb{S}^{n-2})}{24(2\pi)^{n-1}(n^2-1)}
        \Bigl[
            3(n-2)(n^2-n-4)H^2 - 4(n^2-3n-1)R \\
            & \quad + 6n(n-2)(n-1)^2K - 12(n^2-1)q + 12(n^2-1)k^2
        \Bigr], \quad n \geqslant 3,
    \end{align*}
    and
    \begin{align*}
        a_3(x) & = \frac{\Gamma(n-3)\operatorname{vol}(\mathbb{S}^{n-2})}{48(2\pi)^{n-1}(n+3)(n^2-1)}
        \biggl(
            \alpha_{12}H^3 + \alpha_{13}HK + \alpha_{14}HR + \alpha_{15} \sum_{\alpha} \kappa_{\alpha}^3 \\
            & \quad + \alpha_{9}\frac{\partial q}{\partial x_{n}} + \alpha_{10}Hq + \alpha_{11}k^2H
        \biggr), \quad n \geqslant 4.
    \end{align*}
    Here $\alpha_{k}\ (9 \leqslant k \leqslant 11)$ are in Theorem \ref{thm1.1}, and $\alpha_{k}\ (12 \leqslant k \leqslant 15)$ are given by
    \begin{align*}
        & \alpha_{12} = n^5-2n^4-25n^3+12n^2+164n-96,\\
        & \alpha_{13} = 2n(n-2)(3n^4-12n^3-38n^2+108n-21),\\
        & \alpha_{14} = -2(3n^4-13n^3-44n^2+120n+72),\\
        & \alpha_{15} = 4(3n^3-n^2-14n-12).
    \end{align*}
\end{corollary}

\addvspace{5mm}

The main ideas of this paper are as follows. Since the magnetic Dirichlet-to-Neumann map $\mathcal{M}$ is a pseudodifferential operator defined on the boundary $\partial \Omega$ of the manifold $\Omega$, we study this kind of operator by the theory of pseudodifferential operators. There is an effective method (see \cite{Grubb86}, \cite{Hormander85.3}) to calculate its full symbol explicitly. By this method, we flat the boundary and induce a Riemannian metric in a neighborhood of the boundary. Thus, we obtain a local representation for the magnetic Dirichlet-to-Neumann map $\mathcal{M}(u|_{\partial \Omega}) = \bigl(-\frac{\partial u}{\partial x_n} - iA_n u\bigr)\big|_{\partial \Omega}$ in boundary normal coordinates, where $A_{n}$ is the $n$-th component of the vector field $A = \sum_{j} A_j \frac{\partial }{\partial x_j}$. We then look for the factorization
\begin{align*}
    -\mathcal{L}_{g}
    = \frac{\partial^2 }{\partial x_n^2} + B \frac{\partial }{\partial x_n} + C
    = \Bigl(\frac{\partial }{\partial x_n} + B - W\Bigr) \Bigl(\frac{\partial }{\partial x_n} + W\Bigr),
\end{align*}
where $B$, $C$ are two differential operators, and $W$ is a pseudodifferential operator. As a result, we obtain the equation
\begin{align*}
    W^2 - BW - \Bigl[\frac{\partial }{\partial x_n},W\Bigr] + C = 0,
\end{align*}
where $[\cdot,\cdot]$ is the Lie bracket. Next, we solve the corresponding full symbol equation
\begin{align*}
    \sum_{J} \frac{(-i)^{|J|}}{J !} \partial_{\xi^{\prime}}^{J}\!w \, \partial_{x^\prime}^{J}\!w - bw - \frac{\partial w}{\partial x_n} + c = 0,
\end{align*}
where the sum is over all multi-indices $J$. Here $b$, $c$ and $w$ are the full symbols of $B$, $C$ and $W$, respectively. Hence, we get the full symbol of $W$, which implies that $W$ has been obtained on $\partial \Omega$ modulo a smoothing operator. It follows from the results of \cite{Agra87}, \cite{DuisGuill75}, \cite{GrubbSeeley95} and \cite{Liu15} that the trace $\operatorname{Tr} e^{-t \mathcal{M}}$ of the associated heat kernel $\mathcal{K}(t,x^{\prime},y^{\prime})$ admits an asymptotic expansion
\begin{align*}
    \sum_{k=1}^{\infty} e^{-t \lambda_{k}}
    & = \int_{\partial \Omega} \mathcal{K}(t,x^{\prime},x^{\prime}) \,dS \\
    & = \int_{\partial \Omega}
    \bigg\{ 
        \frac{1}{(2\pi)^{n-1}} \int_{T^{*}(\partial \Omega)} e^{i \langle x^{\prime} - x^{\prime},\xi^{\prime} \rangle}
        \bigg[ 
            \frac{i}{2\pi}\int_{\mathcal{C}} e^{-t\tau}(\mathcal{M}-\tau)^{-1} \,d\tau
        \bigg] \,d\xi^{\prime} 
    \bigg\} \,dS \\
    & = \int_{\partial \Omega} 
    \bigg\{ 
        \frac{1}{(2\pi)^{n-1}} \int_{\mathbb{R}^{n-1}}
        \bigg[
            \frac{i}{2\pi}\int_{\mathcal{C}} e^{-t\tau} 
            \bigg( \sum_{j \leqslant -1} s_{j}(x,\xi^{\prime},\tau) \bigg) \,d\tau 
        \bigg] \,d\xi^{\prime}
    \bigg\} \,dS \\
    & \sim \sum_{k=0}^{\infty} a_k t^{-n+k+1} + \sum_{l=0}^{\infty} b_{l} t^l \log t \quad \text{as}\ t \to 0^{+},
\end{align*}
where $\mathcal{C}$ be a contour around the positive real axis, and $\sum_{j \leqslant -1} s_{j}(x,\xi^{\prime},\tau)$ is the full symbol of the pseudodifferential operator $(\mathcal{M}-\tau)^{-1}$. By lots of complicated calculations, we finally establish an effective procedure to calculate all the coefficients $a_0$, $a_1$, $\dots$, $a_{n-1}$ of the heat trace asymptotic expansion and explicitly give the first four coefficients.

This paper is organized as follows. In section 2 we briefly introduce some basic concepts in Riemannian geometry, and give the expression for the magnetic Dirichlet-to-Neumann map $\mathcal{M}$ in boundary normal coordinates. In section 3 we derive the factorization of the magnetic Schr\"{o}dinger operator $\mathcal{L}_{g}$, and get the full symbols of the pseudodifferential operators $\mathcal{M}$ and $(\mathcal{M}-\tau)^{-1}$. Finally, section 4 is devoted to prove Theorem \ref{thm1.1} and Corollary \ref{cor1.3}.

\addvspace{10mm}

\section{Preliminaries}

Let $(\Omega,g)$ be an $n$-dimensional, smooth, compact Riemannian manifold with smooth boundary $\partial \Omega$. In the local coordinates $\{x_i\}_{i=1}^n$, we denote by $\bigl\{\frac{\partial}{\partial x_i}\bigr\}_{i=1}^n$ a natural basis for tangent space $T_x \Omega$ at the point $x \in \Omega$. In what follows, we will let Greek indices run from 1 to $n-1$, whereas Roman indices from 1 to $n$, unless otherwise specified.

In the local coordinates, the Riemannian metric $g$ is given by
\begin{align*}
    g = \sum_{i,j} g_{ij} \,dx_i \otimes dx_j.
\end{align*}
For $X=\sum_{i} X_i \frac{\partial}{\partial x_i}$, $Y=\sum_{i} Y_i \frac{\partial}{\partial x_i} \in \mathfrak{X}(\Omega)$, the inner product with respect to the Riemannian metric $g$ is denoted by
\begin{equation*}
    \langle X,Y \rangle = \sum_{i,j} g_{ij} X_i Y_j.
\end{equation*}
The divergence operator on the Riemannian manifold, in the local coordinates, is denoted by
\begin{equation*}
    \operatorname{div} X 
    = \frac{1}{\sqrt{\operatorname{det}g}} \sum_i \frac{\partial}{\partial x_i} \bigl(\sqrt{\operatorname{det}g}X_{i}\bigr),
\end{equation*}
and the gradient operator is denoted by
\begin{equation*}
    \operatorname{grad} u 
    = \sum_{i,j} g^{ij} \frac{\partial u}{\partial x_i} \frac{\partial}{\partial x_j} \quad \text{for}\ u \in C^{\infty}(\Omega),
\end{equation*}
where $(g^{ij}) = (g_{ij})^{-1}$. Thus, one can define the Laplace-Beltrami operator as
\begin{equation*}
    \Delta_{g} 
    := \operatorname{div} \operatorname{grad} 
    = \frac{1}{\sqrt{\operatorname{det}g}} \sum_{i,j} \frac{\partial}{\partial x_i} \Bigl(\sqrt{\operatorname{det}g} g^{ij} \frac{\partial}{\partial x_j}\Bigr).
\end{equation*}

\addvspace{2mm}

Let $\nabla$ be the Levi-Civita connection. The curvature tensor is defined by
\begin{align*}
    \mathcal{R}(X,Y)Z := [\nabla_{X},\nabla_{Y}]Z - \nabla_{[X,Y]}Z
\end{align*}
for $X,Y,Z \in \mathfrak{X}(\Omega)$. The components of the curvature tensor are given by
\begin{align*}
    \mathcal{R} \Bigl(\frac{\partial }{\partial x_i},\frac{\partial }{\partial x_j}\Bigr) \frac{\partial }{\partial x_k}
    = \sum_{l} R^{l}_{ijk} \frac{\partial }{\partial x_l},
\end{align*}
where
\begin{align*}
    R^{l}_{ijk} 
    = \frac{\partial \Gamma^{l}_{jk}}{\partial x_i} - \frac{\partial \Gamma^{l}_{ik}}{\partial x_j} + \sum_{p} \Gamma^{p}_{jk} \Gamma^{l}_{ip} - \sum_{p} \Gamma^{p}_{ik} \Gamma^{l}_{jp},
\end{align*}
and the Christoffel symbols are given by
\begin{align}\label{chris}
    \Gamma^{i}_{jk} 
    = \frac{1}{2} \sum_{l} g^{kl}
    \Bigl(\frac{\partial g_{jl}}{\partial x_i} + \frac{\partial g_{il}}{\partial x_j} - \frac{\partial g_{ij}}{\partial x_l}\Bigr).
\end{align}
The components of the Riemann curvature tensor are defined by
\begin{align*}
    R_{ijkl}
    := \Big\langle 
    \mathcal{R} \Bigl(\frac{\partial }{\partial x_i},\frac{\partial }{\partial x_j}\Bigr) \frac{\partial }{\partial x_k},\frac{\partial }{\partial x_l} 
    \Big\rangle
    = \sum_{m} g_{lm} R^{m}_{ijk}.
\end{align*}
Some basic symmetries of the Riemann curvature tensor are
\begin{align*}
    R_{ijkl} = -R_{jikl} = -R_{ijlk} = R_{klij}.
\end{align*}
The Ricci tensor is the trace of the Riemann curvature tensor:
\begin{align*}
    \operatorname{Ric} (X,Y) := \operatorname{Tr} (Z \mapsto \mathcal{R} (Z,X)Y)
\end{align*}
for $X,Y,Z \in \mathfrak{X}(\Omega)$. The components of the the Ricci tensor are given by
\begin{align}\label{ricci}
    R_{ij} = \sum_{k,l} g^{kl} R_{iklj}.
\end{align}
For later reference, we write out the covariant derivative of the Ricci tensor as follows:
\begin{align}\label{covariant}
    \nabla_{\frac{\partial }{\partial x_i}} R_{jk}
    = \frac{\partial R_{jk}}{\partial x_i} - \sum_{l} \Gamma^{l}_{ij}R_{lk} - \sum_{l} \Gamma^{l}_{ik}R_{jl}.
\end{align}
The scalar curvature is the trace of the Ricci tensor, that is,
\begin{align}\label{scalar}
    R = \sum_{i,j} g^{ij} R_{ij} = \sum_{i,j,k,l} g^{ij} g^{kl} R_{iklj}.
\end{align}

\addvspace{2mm}

Here we briefly introduce the construction of geodesic coordinates with respect to the boundary (see \cite{LeeUhlm89} or p.~532 of \cite{Taylor96.2}). For each boundary point $x_{0} \in \partial \Omega$, let $\gamma_{x_{0}}:[0,\epsilon)\to \bar{\Omega}$ denote the unit-speed geodesic starting at $x_{0}$ and normal to $\partial \Omega$. If $x^{\prime} = \{x_{1}, \ldots, x_{n-1}\}$ are any local coordinates for $\partial \Omega$ near $x_{0} \in \partial \Omega$, we can extend them smoothly to functions on a neighborhood of $x_{0}$ in $\bar{\Omega}$ by letting them be constant along each normal geodesic $\gamma_{x_{0}}$. If we then define $x_n$ to be the parameter along each $\gamma_{x_{0}}$, it follows easily that $\{x_{1}, \ldots, x_{n}\}$ form coordinates for $\bar{\Omega}$ in some neighborhood of $x_{0}$, which we call the boundary normal coordinates determined by $\{x_{1}, \ldots, x_{n-1}\}$. In these coordinates $x_n>0$ in $\Omega$, and $\partial \Omega$ is locally characterized by $x_n=0$. A standard computation shows that the metric $g$ then has the form
\begin{equation*}
    g = \sum_{\alpha,\beta} g_{\alpha\beta} \,dx_{\alpha} \,dx_{\beta} + (dx_{n})^{2}.
\end{equation*}

Furthermore, we can choose a frame that diagonalizes the second fundamental form $h$, such that the $(\alpha,\beta)$ entry of the matrix of the second fundamental form at the point $x_{0}$ is
\begin{align*}
    h_{\alpha\beta}(x_{0})
    := \Big\langle \nabla_{\frac{\partial }{\partial x_{\alpha}}} \nu, \frac{\partial }{\partial x_{\beta}} \Big\rangle
    = -\Gamma^{\beta}_{n\alpha}
    = -\frac{1}{2}\frac{\partial g_{\alpha\beta}}{\partial x_n}.
\end{align*}
Thus, in these coordinates, for $1 \leqslant j,k \leqslant n$ and $1 \leqslant \alpha,\beta \leqslant n-1$, one has
\begin{align*}
    & g_{jk}(x_{0}) = g^{jk}(x_{0}) = \delta_{jk}, \\
    & \frac{\partial g_{jk}}{\partial x_\alpha}(x_{0}) = \frac{\partial g^{jk}}{\partial x_\alpha}(x_{0}) = 0, \\
    & \frac{\partial g_{\alpha\beta}}{\partial x_n}(x_{0}) = -\frac{\partial g^{\alpha\beta}}{\partial x_n}(x_{0}) = -2\kappa_\alpha \delta_{\alpha\beta},
\end{align*}
where $\delta_{jk}$ is the Kronecker delta. We then have $h_{\alpha\beta}(x_{0}) = \kappa_{\alpha}\delta_{\alpha\beta}$. It follows immediately from \eqref{1.5} that the magnetic Dirichlet-to-Neumann map $\mathcal{M}$ has the form
\begin{equation*}
    \mathcal{M}(f) 
    = \Bigl(-\frac{\partial u}{\partial x_n} - iA_n u\Bigr)\Big|_{\partial \Omega},
\end{equation*}
where $A_{n}$ is the $n$-th component of the vector field $A = \sum_{j} A_j \frac{\partial }{\partial x_j}$.

\addvspace{2mm}

We recall some concepts of pseudodifferential operators and symbols. Assuming $U \subset \mathbb{R}^n$ and $m \in \mathbb{R}$, we define $S^{m}_{1,0} = S^{m}_{1,0}(U,\mathbb{R}^n)$ to consist of $C^{\infty}$-functions $p(x,\xi)$ satisfying for every compact set $Q \subset U$,
\begin{align*}
    |D_{x}^{\beta}D_{\xi}^{\alpha}p(x,\xi)| 
    \leqslant C_{Q,\alpha,\beta} \langle \xi \rangle ^{m-|\alpha|}, \quad x \in Q,\ \xi \in \mathbb{R}^n
\end{align*}
for all $\alpha,\beta \in \mathbb{N}^{n}_{+}$, where $D^{\alpha} = D^{\alpha_1} \cdots D^{\alpha_n}$, $D_{j} = \frac{1}{i} \frac{\partial }{\partial x_{j}}$ and $\langle \xi \rangle = (1 + |\xi|^2)^{1/2}$. The elements of $S^{m}_{1,0}$ are called symbols of order $m$. It is clear that $S^{m}_{1,0}$ is a Fr\'{e}chet space with semi-norms given by
\begin{align*}
    \|p\|_{Q,\alpha,\beta} 
    := \sup_{x \in J}
    \big| \big( D_{x}^{\beta}D_{\xi}^{\alpha}p(x,\xi) \big) (1 + |\xi|)^{-m+|\alpha|} \big|.
\end{align*}
Let $p(x,\xi) \in S^{m}_{1,0}$ and $\hat{u}(\xi) = \int_{\mathbb{R}^n}e^{-iy \cdot \xi} u(y) \,dy$ be the Fourier transform of $u$. A pseudodifferential operator in an open set $U$ is essentially defined by a Fourier integral operator
\begin{align*}
    P(x,D)u(x) 
    = \frac{1}{(2\pi)^n}\int_{\mathbb{R}^n}p(x,\xi)e^{ix \cdot \xi} \hat{u}(\xi) \,d\xi
\end{align*}
for $u \in C^{\infty}_{0}(U)$. In such a case we say the associated operator $P(x,D)$ belongs to $OPS^{m}$. If there are smooth $p_{m−j}(x,\xi)$, homogeneous in $\xi$ of degree $m-j$ for $|\xi| \geqslant 1$, that is, $p_{m-j}(x,r\xi) = r^{m-j}p_{m-j}(x,\xi)$, and if
\begin{align}\label{symbol}
    p(x,\xi) \sim \sum_{j \geqslant 0}p_{m-j}(x,\xi)
\end{align}
in the sense that
\begin{align*}
    p(x,\xi) - \sum_{j = 0}^{N} p_{m-j}(x,\xi) \in S_{1,0}^{m-N-1}
\end{align*}
for all $N$, then we say $p(x,\xi) \in S_{cl}^{m}$, or just $p(x,\xi) \in S^{m}$. We call $p_{m}(x,\xi)$ the principal symbol of $P(x,D)$. We say $P(x,D) \in OPS^{m}$ is elliptic of order $m$ if on each compact $Q \subset U$ there are constants $C_{Q}$ and $r < \infty$ such that
\begin{align*}
    |p(x,\xi)^{-1}| \leqslant C_{Q} \langle \xi \rangle ^{-m}, \quad |\xi| \geqslant r.
\end{align*}

We can now define a pseudodifferential operator on a manifold $\Omega$. Namely
\begin{align*}
    P:C_{0}^{\infty}(\Omega) \to C^{\infty}(\Omega)
\end{align*}
belongs to $OPS^{m}_{1,0}(\Omega)$ if the kernel of $P$ is smooth off the diagonal in $\Omega \times \Omega$ and if for any coordinate neighborhood $U \subset \Omega$ with $\Phi:U \to \mathcal{O}$ a diffeomorphism onto an open subset $\mathcal{O} \subset \mathbb{R}^{n}$, the map $\tilde{P}:C_{0}^{\infty}(\mathcal{O}) \to C^{\infty}(\mathcal{O})$ given by
\begin{align*}
    \tilde{P}u := P(u \circ \Phi) \circ \Phi^{-1}
\end{align*}
belongs to $OPS^{m}_{1,0}(\mathcal{O})$. We refer the reader to \cite{Grubb86}, \cite{Hormander85.3}, \cite{Taylor96.2} and \cite{Taylor81} for more details.

\addvspace{10mm}

\section{Symbols of Pseudodifferential Operators}

In this section, we will derive the factorization of the magnetic Schr\"{o}dinger operator $\mathcal{L}_{g}$, and calculate the full symbols of the pseudodifferential operators $\mathcal{M}$ and $(\mathcal{M}-\tau)^{-1}$.

In boundary normal coordinates, we can write the Laplace-Beltrami operator as
\begin{align*}
    \Delta_g
    = \frac{\partial^2 }{\partial x_n^2} + \frac{1}{2}\sum_{\alpha,\beta} g^{\alpha\beta}\frac{\partial g_{\alpha\beta}}{\partial x_n} \frac{\partial }{\partial x_n} + \sum_{\alpha,\beta} g^{\alpha\beta} \frac{\partial^2}{\partial x_\alpha\partial x_\beta}  + \sum_{\alpha,\beta}
    \biggl(
        \frac{1}{2} g^{\alpha\beta} \sum_{\gamma,\rho} g^{\gamma\rho} \frac{\partial g_{\gamma\rho}}{\partial x_\alpha} + \frac{\partial g^{\alpha\beta}}{\partial x_\alpha}
    \biggr)
    \frac{\partial }{\partial x_\beta}.
\end{align*}
It follows from this and \eqref{1.2} that
\begin{align}\label{2.2}
    -\mathcal{L}_{g} 
    = \frac{\partial^2 }{\partial x_n^2} + B \frac{\partial }{\partial x_n} + C,
\end{align}
where $C = C_2 + C_1 + C_0$, and
\begin{align*}
    & B = \frac{1}{2}\sum_{\alpha,\beta} g^{\alpha\beta} \frac{\partial g_{\alpha\beta}}{\partial x_n} + 2iA_n, \\
    & C_2 = \sum_{\alpha,\beta} g^{\alpha\beta} \frac{\partial^2}{\partial x_\alpha\partial x_\beta}, \\
    & C_1 = \sum_{\alpha,\beta}
    \biggl(
        \frac{1}{2} g^{\alpha\beta} \sum_{\gamma,\rho} g^{\gamma\rho} \frac{\partial g_{\gamma\rho}}{\partial x_\alpha} + \frac{\partial g^{\alpha\beta}}{\partial x_\alpha}
    \biggr)
    \frac{\partial }{\partial x_\beta} + 2i\sum_\alpha A_\alpha\frac{\partial }{\partial x_\alpha}, \\
    & C_0 = -V.
\end{align*}

\addvspace{2mm}

We then derive the microlocal factorization of the magnetic Schr\"{o}dinger operator.
\begin{proposition}\label{prop2.1}
    There exists a pseudodifferential operator $W(x,D_{x^\prime})$ of order one in $x^\prime$ depending smoothly on $x_n$ such that
    \begin{align*}
        -\mathcal{L}_{g} 
        = \Bigl(\frac{\partial }{\partial x_n} + B - W\Bigr)\Bigl(\frac{\partial }{\partial x_n} + W\Bigr)
    \end{align*}
    modulo a smoothing operator.
\end{proposition}
\begin{proof}
    It follows from \eqref{2.2} that
    \begin{align*}
        \frac{\partial^2 }{\partial x_n^2} + B \frac{\partial }{\partial x_n} + C 
        = \Bigl(\frac{\partial }{\partial x_n} + B - W\Bigr)\Bigl(\frac{\partial }{\partial x_n} + W\Bigr),
    \end{align*}
    that is,
    \begin{align}\label{2.4}
        W^2 - BW - \Bigl[\frac{\partial }{\partial x_n},W\Bigr] + C = 0,
    \end{align}
    where the commutator $\big[\frac{\partial }{\partial x_n},W\big]$ is defined by
    \begin{align*}
        \Bigl[\frac{\partial }{\partial x_n},W\Bigr]v
        & = \frac{\partial }{\partial x_n}(Wv) - W\frac{\partial }{\partial x_n}v \\
        & = \frac{\partial W}{\partial x_n}v
    \end{align*}
    for any smooth function $v$.

    Recall that if $F$ and $G$ are two pseudodifferential operators with full symbols $f$ and $g$, respectively, then the full symbol $\sigma(FG)$ of $FG$ is given by (see p.~11 of  \cite{Taylor96.2}, or p.~71 of \cite{Hormander85.3}, see also \cite{Grubb86}, \cite{Treves80})
    \begin{align*}
        \sigma(FG)\sim \sum_{J} \frac{(-i)^{|J|}}{J !} \partial_{\xi}^{J}\!f \, \partial_{x}^{J}\!g,
    \end{align*}
    where the sum is over all multi-indices $J$. Let $w = w(x,\xi^{\prime})$ be the full symbol of $W$, from \eqref{symbol} we write
    \begin{align*}
        w(x,\xi^{\prime}) \sim \sum_{j\leqslant 1} w_j(x,\xi^{\prime}),
    \end{align*}
    with $w_j(x,\xi^{\prime})$ homogeneous of degree $j$ in $\xi^{\prime}$. Let $b(x,\xi^{\prime})$ and $c(x,\xi^{\prime}) = c_2(x,\xi^{\prime}) + c_1(x,\xi^{\prime}) + c_0(x,\xi^{\prime})$ be the full symbols of $B$ and $C$, respectively. We then obtain
    \begin{align*}
        & b(x,\xi^{\prime}) = \frac{1}{2}\sum_{\alpha,\beta} g^{\alpha\beta} \frac{\partial g_{\alpha\beta}}{\partial x_n} + 2iA_n,\\
        & c_2(x,\xi^{\prime}) = -\sum_{\alpha,\beta} g^{\alpha\beta} \xi_\alpha \xi_\beta := -|\xi^{\prime}|^2,\\
        & c_1(x,\xi^{\prime}) = i\sum_{\alpha,\beta}
        \biggl(
            \frac{1}{2} g^{\alpha\beta} \sum_{\gamma,\rho} g^{\gamma\rho} \frac{\partial g_{\gamma\rho}}{\partial x_\alpha} + \frac{\partial g^{\alpha\beta}}{\partial x_\alpha}
        \biggr)
        \xi_\beta - 2\sum_\alpha A_\alpha \xi_\alpha, \\
        & c_0(x,\xi^{\prime}) = -V.
    \end{align*}
    Noting that $\partial_{\xi^{\prime}}^{J} b(x,\xi^{\prime}) = 0$ for all $|J|>0$, we conclude that the full symbol equation of \eqref{2.4} is
    \begin{equation}\label{2.5}
        \sum_{J} \frac{(-i)^{|J|}}{J !} \partial_{\xi^{\prime}}^{J}\!w \, \partial_{x^\prime}^{J}\!w - bw - \frac{\partial w}{\partial x_n} + c = 0.
    \end{equation}

    We shall determine the $w_j$ recursively so that \eqref{2.5} holds modulo $S^{-\infty}$. Grouping the homogeneous terms of degree two in \eqref{2.5}, which implies
    \begin{align*}
        w_1^2 + c_2 = 0,
    \end{align*}
    so we can choose
    \begin{align}\label{2.6}
        w_1 = |\xi^{\prime}|.
    \end{align}
    The terms of degree one in \eqref{2.5} are
    \begin{align*}
        2w_1w_0 - i\sum_\alpha\frac{\partial w_1}{\partial \xi_\alpha}\frac{\partial w_1}{\partial x_\alpha} - bw_1 - \frac{\partial w_1}{\partial x_n} + c_1 = 0,
    \end{align*}
    so that
    \begin{align}\label{2.7}
        w_0 
        = \frac{1}{2}w_1^{-1}
        \biggl(
            i\sum_\alpha\frac{\partial w_1}{\partial \xi_\alpha}\frac{\partial w_1}{\partial x_\alpha} + bw_1 + \frac{\partial w_1}{\partial x_n} - c_1
        \biggr).
    \end{align}
    The terms of degree zero in \eqref{2.5} are
    \begin{align*}
        2w_1w_{-1} + w_0^2 - i\sum_\alpha
        \Bigl(
            \frac{\partial w_1}{\partial \xi_\alpha}\frac{\partial w_0}{\partial x_\alpha} + \frac{\partial w_0}{\partial \xi_\alpha}\frac{\partial w_1}{\partial x_\alpha}
        \Bigr)
        - \frac{1}{2}\sum_{\alpha,\beta} \frac{\partial^2 w_1}{\partial \xi_{\alpha}\partial \xi_{\beta}}\frac{\partial^2 w_1}{\partial x_{\alpha}\partial x_{\beta}} - bw_0 - \frac{\partial w_0}{\partial x_n}  + c_0 = 0,
    \end{align*}
    i.e.,
    \begin{align}\label{2.8}
        w_{-1} 
        & = \frac{1}{2}w_1^{-1}
        \bigg[
            -w_0^2 + i\sum_\alpha
            \Bigl(
                \frac{\partial w_1}{\partial \xi_\alpha}\frac{\partial w_0}{\partial x_\alpha} + \frac{\partial w_0}{\partial \xi_\alpha}\frac{\partial w_1}{\partial x_\alpha}
            \Bigr) 
            + \frac{1}{2}\sum_{\alpha,\beta} \frac{\partial^2 w_1}{\partial \xi_{\alpha}\partial \xi_{\beta}}\frac{\partial^2 w_1}{\partial x_{\alpha}\partial x_{\beta}} \\
            & \quad + bw_0 + \frac{\partial w_0}{\partial x_n} - c_0 \notag
        \bigg].
    \end{align}
    The terms of degree $-1$ in \eqref{2.5} are
    \begin{align*}
        & 2w_1w_{-2} + 2w_0w_{-1} - i\sum_\alpha
        \Bigl(
            \frac{\partial w_1}{\partial \xi_\alpha}\frac{\partial w_{-1}}{\partial x_\alpha}+\frac{\partial w_0}{\partial \xi_\alpha}\frac{\partial w_0}{\partial x_\alpha} +\frac{\partial w_{-1}}{\partial \xi_\alpha}\frac{\partial w_1}{\partial x_\alpha}
        \Bigr) \\
        & - \frac{1}{2} \sum_{\alpha,\beta}
        \Bigl(
            \frac{\partial^2 w_1}{\partial \xi_{\alpha}\partial \xi_{\beta}} \frac{\partial^2 w_0}{\partial x_{\alpha}\partial x_{\beta}} + \frac{\partial^2 w_0}{\partial \xi_{\alpha}\partial \xi_{\beta}} \frac{\partial^2 w_1}{\partial x_{\alpha}\partial x_{\beta}}
        \Bigr) + \frac{i}{6}\sum_{\alpha,\beta,\gamma} \frac{\partial^3 w_1}{\partial\xi_{\alpha}\partial\xi_{\beta}\partial\xi_{\gamma}}\frac{\partial^3 w_1}{\partial x_{\alpha}\partial x_{\beta}\partial x_{\gamma}} \\
        & - bw_{-1} - \frac{\partial w_{-1}}{\partial x_n} = 0,
    \end{align*}
    we obtain
    \begin{align}\label{2.9}
        w_{-2} 
        & = \frac{1}{2}w_1^{-1}
        \biggl[
            -2w_0w_{-1} + i\sum_\alpha
            \Bigl(
                \frac{\partial w_1}{\partial \xi_\alpha}\frac{\partial w_{-1}}{\partial x_\alpha}+\frac{\partial w_0}{\partial \xi_\alpha}\frac{\partial w_0}{\partial x_\alpha} +\frac{\partial w_{-1}}{\partial \xi_\alpha}\frac{\partial w_1}{\partial x_\alpha}
            \Bigr) \\
            & \quad + \frac{1}{2} \sum_{\alpha,\beta}
            \Bigl(
                \frac{\partial^2 w_1}{\partial \xi_{\alpha}\partial \xi_{\beta}} \frac{\partial^2 w_0}{\partial x_{\alpha}\partial x_{\beta}} + \frac{\partial^2 w_0}{\partial \xi_{\alpha}\partial \xi_{\beta}} \frac{\partial^2 w_1}{\partial x_{\alpha}\partial x_{\beta}}
            \Bigr)
            - \frac{i}{6}\sum_{\alpha,\beta,\gamma} \frac{\partial^3 w_1}{\partial\xi_{\alpha}\partial\xi_{\beta}\partial\xi_{\gamma}}\frac{\partial^3 w_1}{\partial x_{\alpha}\partial x_{\beta}\partial x_{\gamma}} \notag\\
            & \quad + bw_{-1} + \frac{\partial w_{-1}}{\partial x_n}
        \biggr]. \notag
    \end{align}
    Proceeding recursively, for the terms of degree $-m\ (m \geqslant 2)$, we have
    \begin{align*}
        2w_1w_{-1-m} + \sum_
        {\mbox
            {\tiny
                $\begin{array}{c}
                    -m \leqslant j,k \leqslant 1 \\
                    |J| = j + k + m
                \end{array}$
            }
        }
        \frac{(-i)^{|J|}}{J!} \partial_{\xi^{\prime}}^{J}\!w_j \, \partial_{x^\prime}^{J}\!w_k - bw_{-m} - \frac{\partial w_{-m}}{\partial x_n} = 0,
    \end{align*}
    namely
    \begin{align}\label{2.9.1}
        w_{-1-m} = -\frac{1}{2}w_1^{-1}
        \Biggl(
            \sum_
            {\mbox
                {\tiny
                    $\begin{array}{c}
                        -m \leqslant j,k \leqslant 1 \\
                        |J| = j + k + m
                    \end{array}$
                }
            }
            \frac{(-i)^{|J|}}{J !} \partial_{\xi^{\prime}}^{J}\!w_j \, \partial_{x^\prime}^{J}\!w_k - bw_{-m} - \frac{\partial w_{-m}}{\partial x_n}
        \Biggr), \quad m \geqslant 2.
    \end{align}
\end{proof}

\addvspace{2mm}

From Proposition \ref{prop2.1}, we get the full symbol $w(x,\xi) \sim \sum_{j\leqslant 1} w_j(x,\xi)$ of the pseudodifferential operator $W$. This implies that $W$ has been obtained on $\partial \Omega$ modulo a smoothing operator.
\begin{proposition}
    In the boundary normal coordinates, the magnetic Dirichlet-to-Neumann map $\mathcal{M}$ can be represented as
    \begin{align}\label{2.10}
        \mathcal{M}(f) = (W - iA_n)u|_{\partial \Omega}
    \end{align}
    modulo a smoothing operator.
\end{proposition}
\begin{proof}
    This proof is similar to the proof of Proposition 1.2 in \cite{LeeUhlm89}.
\end{proof}

It follows from \eqref{2.10} that the full symbol $\sigma(\mathcal{M})$ of the magnetic Dirichlet-to-Neumann map $\mathcal{M}$ has the form
\begin{equation}\label{2.11}
    \sigma(\mathcal{M}) \sim w_1 + (w_0 - iA_n) + \sum_{j \leqslant -1} w_j.
\end{equation}
We apply the methods of Grubb \cite{Grubb86} and Seely \cite{Seeley67}. Let $S(\tau)$ be a two-sided parametrix for $\mathcal{M} - \tau$, that is, $S(\tau)$ is a pseudodifferential operator of order $-1$ with parameter $\tau$ for which
\begin{align*}
    & S(\tau)(\mathcal{M} - \tau) = 1 \mod OPS^{-\infty}, \\
    & (\mathcal{M} - \tau)S(\tau) = 1 \mod OPS^{-\infty}.
\end{align*}
Let $s(x,\xi^{\prime},\tau)$ be the full symbol of $S(\tau)$, we write
\begin{align*}
    s(x,\xi^{\prime},\tau) \sim \sum_{j \leqslant -1} s_j(x,\xi^{\prime},\tau).
\end{align*}
Therefore, we obtain
\begin{align}\label{2.14}
    s_{-1} = (w_1-\tau)^{-1},
\end{align}
and
\begin{align}\label{2.14.1}
    s_{-1-m} = -s_{-1} \sum_
    {\mbox
        {\tiny
            $\begin{array}{c}
                -m \leqslant j \leqslant 1 \\
                -m \leqslant k \leqslant -1 \\
                |J| = j + k + m
            \end{array}$
        }
    }
    \frac{(-i)^{|J|}}{J !} \partial_{\xi^{\prime}}^{J}\! \sigma_j(\mathcal{M}) \, \partial_{x^\prime}^{J}\! s_k, \quad m \geqslant 1.
\end{align}
For later reference, we write out the expressions for $s_{-2}$, $s_{-3}$ and $s_{-4}$ as follows:
\begin{align}
    s_{-2} & = -s_{-1}
    \biggl[
        (w_0 - iA_n)s_{-1} - i\sum_\alpha \frac{\partial w_1}{\partial \xi_\alpha} \frac{\partial s_{-1}}{\partial x_\alpha}
    \biggr], \label{2.16}\\[1mm]
    s_{-3} & = -s_{-1}
    \bigg[
        (w_0 - iA_n)s_{-2} + w_{-1}s_{-1} - i \sum_\alpha
        \Bigl(
            \frac{\partial w_1}{\partial \xi_\alpha} \frac{\partial s_{-2}}{\partial x_\alpha} + \frac{\partial w_0}{\partial \xi_\alpha} \frac{\partial s_{-1}}{\partial x_\alpha}
        \Bigr) \label{2.17}\\
        & \quad - \frac{1}{2} \sum_{\alpha,\beta} \frac{\partial^2 w_1}{\partial \xi_{\alpha}\partial \xi_{\beta}} \frac{\partial^2 s_{-1}}{\partial x_{\alpha}\partial x_{\beta}}
    \bigg], \notag
\end{align}
and
\begin{align}\label{2.18}
    s_{-4} & = -s_{-1}
    \bigg[
        (w_0 - iA_n)s_{-3} + w_{-1}s_{-2} + w_{-2}s_{-1} - i\sum_\alpha
        \Bigl(
            \frac{\partial w_1}{\partial \xi_\alpha} \frac{\partial s_{-3}}{\partial x_\alpha} + \frac{\partial w_0}{\partial \xi_\alpha} \frac{\partial s_{-2}}{\partial x_\alpha} \\
            & \quad + \frac{\partial w_{-1}}{\partial \xi_\alpha} \frac{\partial s_{-1}}{\partial x_\alpha}
        \Bigr)
        - \frac{1}{2} \sum_{\alpha,\beta}
        \Bigl(
            \frac{\partial^2 w_1}{\partial \xi_{\alpha}\partial \xi_{\beta}} \frac{\partial^2 s_{-2}}{\partial x_{\alpha}\partial x_{\beta}} + \frac{\partial^2 w_0}{\partial \xi_{\alpha}\partial \xi_{\beta}} \frac{\partial^2 s_{-1}}{\partial x_{\alpha}\partial x_{\beta}}
        \Bigr) \notag\\
        & \quad + \frac{i}{6} \sum_{\alpha,\beta,\gamma} \frac{\partial^3 w_1}{\partial\xi_{\alpha}\partial\xi_{\beta}\partial\xi_{\gamma}} \frac{\partial^3 s_{-1}}{\partial x_{\alpha}\partial x_{\beta}\partial x_{\gamma}}
    \bigg]. \notag
\end{align}

\addvspace{10mm}

\section{Calculations of the Coefficients of the Heat Trace Asymptotic Expansion}

In this section, inspired by the methods in \cite{Grubb86}, \cite{Liu15}, \cite{Liu19}, \cite{Seeley67} and \cite{Seeley69}, we will establish an effective procedure to calculate the first $n-1$ coefficients of the heat trace asymptotic expansion associated with the magnetic Dirichlet-to-Neumann map.

According to the theory of elliptic equations (see \cite{Morrey66}, \cite{Morrey58}, \cite{Stewart74}), we see that the magnetic Dirichlet-to-Neumann map $\mathcal{M}$ associated with the magnetic Schr\"{o}dinger operator can generate a strongly continuous generalized heat semigroup $e^{-t\mathcal{M}}$ in a suitable space defined on the boundary $\partial \Omega$. Furthermore, there exists a parabolic kernel $\mathcal{K}(t,x^{\prime},y^{\prime})$ such that (see \cite{GimpGrubb14})
\begin{align*}
    e^{-t\mathcal{M}} \phi(x^{\prime})
    = \int_{\partial \Omega} \mathcal{K}(t,x^{\prime},y^{\prime}) \,dS
\end{align*}
for $\phi \in H^{1}(\partial \Omega)$. Let $\{u_k\}_{k \geqslant 1}$ be the orthonormal eigenfunctions of the magnetic Dirichlet-to-Neumann map $\mathcal{M}$ corresponding to the eigenvalues $\{\lambda_k\}_{k \geqslant 1}$, then the kernel $\mathcal{K}(t,x^{\prime},y^{\prime})$ is given by
\begin{align*}
    \mathcal{K}(t,x^{\prime},y^{\prime})
    & = e^{-t\mathcal{M}} \delta(x^{\prime} - y^{\prime}) \\
    & = \sum_{k=1}^{\infty} e^{-t\lambda_k} u_k(x^{\prime}) \otimes u_k(y^{\prime}).
\end{align*}
This implies that the integral of the trace of the kernel $\mathcal{K}(t,x^{\prime},y^{\prime})$, i.e.,
\begin{align*}
    \int_{\partial \Omega} \mathcal{K}(t,x^{\prime},x^{\prime}) \,dS
    = \sum_{k=1}^{\infty} e^{-t\lambda_k}
\end{align*}
is actually a spectral invariants. On the other hand, the semigroup $e^{-t\mathcal{M}}$ can also be represented as
\begin{align*}
    e^{-t\mathcal{M}}
    = \frac{i}{2\pi}\int_{\mathcal{C}} e^{-t\tau}(\mathcal{M}-\tau)^{-1} \,d\tau
\end{align*}
where $\mathcal{C}$ is a suitable curve in the complex plane in the positive direction around the spectrum of $\mathcal{M}$, that is, $\mathcal{C}$ is a contour around the positive real axis. It follows that
\begin{align*}
    \mathcal{K}(t,x^{\prime},y^{\prime})
    & = e^{-t\mathcal{M}} \delta(x^{\prime} - y^{\prime}) \\
    & = \frac{1}{(2\pi)^{n-1}} \int_{T^{*}(\partial \Omega)} e^{i \langle x^{\prime} - y^{\prime},\xi^{\prime} \rangle} 
    \bigg[ 
        \frac{i}{2\pi}\int_{\mathcal{C}} e^{-t\tau}(\mathcal{M}-\tau)^{-1} \,d\tau
    \bigg] \,d\xi^{\prime} \\
    & = \frac{1}{(2\pi)^{n-1}} \int_{\mathbb{R}^{n-1}} e^{i \langle x^{\prime} - y^{\prime},\xi^{\prime} \rangle}
    \bigg[
        \frac{i}{2\pi}\int_{\mathcal{C}} e^{-t\tau} 
        \bigg( \sum_{j \leqslant -1} s_{j}(x,\xi^{\prime},\tau) \bigg) \,d\tau 
    \bigg] \,d\xi^{\prime},
\end{align*}
such that the trace of the kernel $\mathcal{K}(t,x^{\prime},y^{\prime})$ is
\begin{align*}
    \mathcal{K}(t,x^{\prime},x^{\prime})
    = \frac{1}{(2\pi)^{n-1}} \int_{\mathbb{R}^{n-1}}
    \bigg[
        \frac{i}{2\pi} \int_{\mathcal{C}} e^{-t\tau} 
        \bigg( \sum_{j \leqslant -1} s_{j}(x,\xi^{\prime},\tau) \bigg) \,d\tau 
    \bigg] \,d\xi^{\prime}.
\end{align*}
Therefore,
\begin{align*}
    \sum_{k=1}^{\infty} e^{-t\lambda_k}
    = \int_{\partial \Omega} 
    \bigg\{ 
        \frac{1}{(2\pi)^{n-1}} \int_{\mathbb{R}^{n-1}}
        \bigg[
            \frac{i}{2\pi}\int_{\mathcal{C}} e^{-t\tau} 
            \bigg( \sum_{j \leqslant -1} s_{j}(x,\xi^{\prime},\tau) \bigg) \,d\tau 
        \bigg] \,d\xi^{\prime}
    \bigg\} \,dS.
\end{align*}
We will calculate the asymptotic expansion of the trace of the semigroup $e^{-t\mathcal{M}}$ as $t \to 0^{+}$. More precisely, we will figure out that
\begin{equation}\label{3.1}
    a_{k}(x) 
    = \frac{i}{(2 \pi)^{n}}\int_{\mathbb{R}^{n-1}}\int_{\mathcal{C}} e^{-\tau}s_{-1-k}(x,\xi^{\prime},\tau) \,d\tau \,d\xi^{\prime}
\end{equation}
for $0 \leqslant k \leqslant n-1$.

\addvspace{2mm}

It is convenient to calculate the coefficients of the asymptotic expansion by the following lemma:
\begin{lemma}\label{lem3.1}
    (\romannumeral 1) For any $k \geqslant 1$, we have
    \begin{equation*}
        \frac{1}{2\pi i} \int_{\mathcal{C}} s_{-1}^k e^{-\tau} \,d\tau 
        = -\frac{1}{(k-1)!}e^{-w_1}.
    \end{equation*}

    (\romannumeral 2) For any real $k$, we have the following integrals:
    \begin{align*}
        \int_{\mathbb{R}^{n-1}} w_1^{k} e^{-w_1} \xi_{\alpha}^m \,d\xi^{\prime} =
        \left\{
            \begin{array}{ll}
                \Gamma(n+k-1) \operatorname{vol}(\mathbb{S}^{n-2}), \quad & m = 0, \\
                0, \quad & m = 1,
            \end{array}
        \right.
    \end{align*}
    \begin{align*}
        \int_{\mathbb{R}^{n-1}} w_1^{k-2} e^{-w_1} \xi_{\alpha} \xi_{\beta} \,d\xi^{\prime} =
        \left\{
            \begin{array}{ll}
                \displaystyle \frac{1}{n-1} \Gamma(n+k-1) \operatorname{vol}(\mathbb{S}^{n-2}), \quad & \alpha = \beta, \\[3mm]
                0, \quad & \alpha \neq \beta,
            \end{array}
        \right.
    \end{align*}
    and
    \begin{align*}
        \int_{\mathbb{R}^{n-1}} w_1^{k-4} e^{-w_1} \xi_{\alpha}^2 \xi_{\beta}^2 \,d\xi^{\prime} =
        \left\{
            \begin{array}{ll}
                \displaystyle \frac{3}{n^{2}-1} \Gamma(n+k-1) \operatorname{vol}(\mathbb{S}^{n-2}), \quad & \alpha = \beta, \\[3mm]
                \displaystyle \frac{1}{n^{2}-1} \Gamma(n+k-1) \operatorname{vol}(\mathbb{S}^{n-2}), \quad & \alpha \neq \beta.
            \end{array}
        \right.
    \end{align*}
\end{lemma}

\addvspace{5mm}

\begin{proof}
    The proof of the first result is a simple computation with calculus of residues. The second result can be obtained by applying the spherical coordinates transform (or see \cite{Liu15}).
\end{proof}

\addvspace{5mm}

\begin{proof}[Proof of Theorem \ref{thm1.1}]
    We divide this proof into several steps.
    
    \addvspace{2mm}

    Step 1. It follows from \eqref{2.6} and \eqref{2.14} that
    \begin{equation*}
        s_{-1} = (w_1 - \tau)^{-1} = (|\xi^{\prime}| - \tau)^{-1}.
    \end{equation*}
    Therefore, using formula \eqref{3.1} and Lemma \ref{lem3.1}, we immediately get
    \begin{align*}
        a_{0}(x)
        & = \frac{i}{(2 \pi)^{n}}\int_{\mathbb{R}^{n-1}}\int_{\mathcal{C}} e^{-\tau}s_{-1} \,d\tau \,d\xi^{\prime} \\
        & = \frac{\Gamma(n-1)\operatorname{vol}(\mathbb{S}^{n-2})}{(2\pi)^{n-1}}, \quad n \geqslant 1.
    \end{align*}
    
    \addvspace{2mm}

Step 2. Note that $\frac{\partial s_{-1}}{\partial x_\alpha}(x_{0}) = \frac{\partial w_1}{\partial x_\alpha}(x_{0}) = 0$ in the boundary normal coordinates. Thus,
\begin{align*}
    s_{-2} = -s_{-1}^2 (w_0 - iA_n).
\end{align*}
According to \eqref{2.7}, we see that
\begin{align}\label{3.12}
    w_0
    & = \frac{i}{4}w_1^{-3} \sum_{\alpha,\beta,\gamma} \frac{\partial g^{\beta\gamma}}{\partial x_\alpha} \xi_\alpha\xi_\beta\xi_\gamma + \frac{1}{4} \sum_{\alpha,\beta} g^{\alpha\beta}\frac{\partial g_{\alpha\beta}}{\partial x_n} + iA_n + \frac{1}{4}w_1^{-2} \sum_{\alpha,\beta} \frac{\partial g^{\alpha\beta}}{\partial x_n}\xi_\alpha\xi_\beta \\
    & \quad - \frac{i}{2}w_1^{-1} \sum_{\alpha,\beta}
    \biggl(
        \frac{1}{2}g^{\alpha\beta} \sum_{\gamma,\rho} g^{\gamma\rho}\frac{\partial g_{\gamma\rho}}{\partial x_\alpha} + \frac{\partial g^{\beta\gamma}}{\partial x_\alpha}
    \biggr)
    \xi_\beta + w_1^{-1} \sum_\alpha A_\alpha\xi_\alpha. \notag
\end{align}
In the boundary coordinates, we find that
\begin{align}\label{3.12.1}
    & b(x_{0}) = -H + 2iA_n,
\end{align}
and
\begin{align*}
    c_1(x_{0}) = -2\sum_\alpha A_\alpha\xi_\alpha.
\end{align*}
Hence, it follows from \eqref{2.11} that
\begin{align}\label{3.13}
    (w_0 - iA_n)(x_{0}) 
    = \frac{1}{2}w_1^{-2} \sum_\alpha \kappa_\alpha\xi_{\alpha}^2 - \frac{1}{2}H + w_1^{-1}\sum_\alpha A_\alpha\xi_\alpha.
\end{align}
We then obtain
\begin{align*}
    s_{-2} \sim -\frac{1}{2}s_{-1}^2
    \biggl(
        w_1^{-2} \sum_\alpha \kappa_\alpha\xi_{\alpha}^2 - H
    \biggr).
\end{align*}
Therefore, by applying formula \eqref{3.1} and Lemma \ref{lem3.1}, we get
\begin{align*}
    a_{1}(x)
    & = \frac{i}{(2 \pi)^{n}}\int_{\mathbb{R}^{n-1}}\int_{\mathcal{C}} e^{-\tau}s_{-2} \,d\tau \,d\xi^{\prime} \\
    & = \frac{(n-2)\Gamma(n-1)\operatorname{vol}(\mathbb{S}^{n-2})}{2(2\pi)^{n-1}(n-1)}H, \quad n \geqslant 2.
\end{align*}

\addvspace{2mm}

Step 3. From \eqref{2.14} and \eqref{2.16}, we compute
\begin{align*}
    \frac{\partial s_{-1}}{\partial x_\alpha} 
    = -s_{-1}^2\frac{\partial w_1}{\partial x_\alpha},
\end{align*}
and
\begin{align}\label{3.17}
    \frac{\partial s_{-2}}{\partial x_\alpha}
    & = 2s_{-1}^3(w_0 - iA_n)\frac{\partial w_1}{\partial x_\alpha} - s_{-1}^2\frac{\partial (w_0 - iA_n)}{\partial x_\alpha} + 3is_{-1}^4\frac{\partial w_1}{\partial x_\alpha} \sum_\beta \frac{\partial w_1}{\partial \xi_\beta}\frac{\partial w_1}{\partial x_\beta} \\
    & \quad - is_{-1}^3 \sum_\beta \Bigl(\frac{\partial^2 w_1}{\partial\xi_\beta\partial x_\alpha} \frac{\partial w_1}{\partial x_\beta} + \frac{\partial w_1}{\partial \xi_\beta}\frac{\partial^2 w_1}{\partial x_{\alpha}\partial x_{\beta}}\Bigr). \notag
\end{align}
According to \eqref{2.17}, we conclude that
\begin{align}\label{3.18}
    s_{-3}
    & = s_{-1}^3(w_0 - iA_n)^2 + 3is_{-1}^4(w_0 - iA_n) \sum_\alpha \frac{\partial w_1}{\partial \xi_\alpha} \frac{\partial w_1}{\partial x_\alpha} - s_{-1}^2w_{-1} - is_{-1}^3 \sum_\alpha \frac{\partial w_1}{\partial \xi_\alpha} \frac{\partial (w_0 - iA_n)}{\partial x_\alpha} \\
    & \quad - 3s_{-1}^5\biggl(\sum_\alpha \frac{\partial w_1}{\partial \xi_\alpha}\frac{\partial w_1}{\partial x_\alpha}\biggr)^2 + s_{-1}^4 \sum_{\alpha,\beta} \frac{\partial w_1}{\partial \xi_\alpha} \frac{\partial^2 w_1}{\partial\xi_\beta\partial x_\alpha} \frac{\partial w_1}{\partial x_\beta} + s_{-1}^4 \sum_{\alpha,\beta} \frac{\partial w_1}{\partial \xi_\alpha} \frac{\partial w_1}{\partial \xi_\beta} \frac{\partial^2 w_1}{\partial x_{\alpha}\partial x_{\beta}} \notag\\
    & \quad - is_{-1}^3 \sum_\alpha \frac{\partial w_0}{\partial \xi_\alpha} \frac{\partial w_1}{\partial x_\alpha} + s_{-1}^4 \sum_{\alpha,\beta} \frac{\partial^2 w_1}{\partial \xi_{\alpha}\partial \xi_{\beta}} \frac{\partial w_1}{\partial x_\alpha} \frac{\partial w_1}{\partial x_\beta} - \frac{1}{2}s_{-1}^3 \sum_{\alpha,\beta} \frac{\partial^2 w_1}{\partial \xi_{\alpha}\partial \xi_{\beta}} \frac{\partial^2 w_1}{\partial x_{\alpha}\partial x_{\beta}}. \notag
\end{align}

In order to simplify the calculations and to investigate the roles of the magnetic potential $A$ and the electric potential $q$, we consider the terms which only involve $A$ or $q$ in an expression $f$ at the point $x_{0}$, and denote by $f(A,q)$ the sum of these terms. Thus, it follows from formula \eqref{3.1} that
\begin{equation}\label{3.9}
    a_{k}(A,q) 
    = \frac{i}{(2 \pi)^{n}}\int_{\mathbb{R}^{n-1}}\int_{\mathcal{C}} e^{-\tau}s_{-1-k}(A,q) \,d\tau \,d\xi^{\prime}
\end{equation}
for $0 \leqslant k \leqslant n-1$.
Therefore, we get
\begin{align}\label{3.10}
    a_{k}(x) = \tilde{a}_{k}(x) + a_{k}(A,q),
\end{align}
where $\tilde{a}_{k}(x)$ are the spectral invariants of the classical Dirichlet-to-Neumann map associated with the Laplace-Beltrami operator $\Delta_{g}$. For later reference, we list the expressions for $\tilde{a}_2(x)$ and $\tilde{a}_3(x)$ as follows (see \cite{Liu15}):
\begin{align}\label{3.30}
    \tilde{a}_2(x) 
    & = \frac{\Gamma(n-2)\operatorname{vol}(\mathbb{S}^{n-2})}{8(2\pi)^{n-1}}
    \biggl[
        \frac{4-n}{3(n-1)}R + \frac{n-2}{n-1}\tilde{R} + \frac{n^3-4n^2+n+8}{n^2-1}H^2 \\[1mm]
        & \quad + \frac{n(n-3)}{n^2-1} \sum_{\alpha} \kappa_{\alpha}^2
    \biggr], \quad n \geqslant 3; \notag
\end{align}
and
\begin{align}\label{3.31}
    \tilde{a}_3(x)
    & = \frac{\Gamma(n-3)\operatorname{vol}(\mathbb{S}^{n-2})}{8(n-1)(2\pi)^{n-1}}
    \biggl[
        \frac{n^3-6n^2+2n+14}{2(n+1)}H\tilde{R} - \frac{3n^3-20n^2+12n+42}{6(n+1)}HR \\[1mm]
        & \quad + \frac{n^5-5n^4-10n^3+52n^2+2n-114}{6(n+1)(n+3)}H^3 + \frac{n^4-n^3-12n^2+22n+6}{2(n+1)(n+3)}H\sum_{\alpha} \kappa_{\alpha}^2 \notag\\[1mm]
        & \quad - \frac{4(n-3)(n-2)}{3(n+1)(n+3)}\sum_{\alpha} \kappa_{\alpha}^3 + \frac{2(n^2-3n+1)}{n+1}\sum_{\alpha} \kappa_{\alpha}\tilde{R}_{\alpha\alpha} - \frac{2n(3n-8)}{3(n+1)}\sum_{\alpha} \kappa_{\alpha}R_{\alpha\alpha} \notag\\[1mm]
        & \quad + (n-2) \nabla_{\frac{\partial }{\partial x_n}}\tilde{R}_{nn}
    \biggr], \quad n \geqslant 4. \notag
\end{align}
Noting that the dimension of the domain $\Omega$ is $n+1$ in \cite{Liu15}, here we have changed it to $\operatorname{dim} \Omega = n$ in the expressions above.

\addvspace{2mm}

It follows from \eqref{3.18} that the terms which only involve $A$ or $q$ in $s_{-3}$ are
\begin{equation}\label{3.19}
    s_{-3}(A,q) 
    = s_{-1}^3(w_0 - iA_n)^2 - s_{-1}^2w_{-1} - is_{-1}^3 \sum_\alpha \frac{\partial w_1}{\partial \xi_\alpha} \frac{\partial (w_0 - iA_n)}{\partial x_\alpha}.
\end{equation}
It follows from \eqref{2.8} that
\begin{align*}
    w_{-1}(A,q)
    & = \frac{1}{2}w_1^{-1}
    \bigg[
        -(w_0 - iA_n)^2 + i\sum_\alpha \frac{\partial w_1}{\partial \xi_\alpha}\frac{\partial (w_0 - iA_n)}{\partial x_\alpha} - H(w_0 - iA_n) + \frac{\partial (w_0 - iA_n)}{\partial x_n} \\
        & \quad - \sum_\alpha \frac{\partial w_1}{\partial \xi_\alpha}\frac{\partial A_{n}}{\partial x_{\alpha}} + E
    \bigg],
\end{align*}
where
\begin{equation}\label{3.25}
    E = A_{n}^2 + iA_nb + i\frac{\partial A_{n}}{\partial x_n} - c_0.
\end{equation}
Recall that (see \eqref{1.3})
\begin{align}\label{3.26}
    V & = \sum_{j,k} g_{jk}A_jA_k - i
    \biggl(
        \sum_j \frac{\partial A_j}{\partial x_j} + \frac{1}{2}\sum_{j,k,l} g^{jk} \frac{\partial g_{jk}}{\partial x_l} A_l
    \biggr) + q - k^2.
\end{align}
Thus, in the boundary normal coordinates, at the point $x_{0}$, we get
\begin{align*}
    V(x_{0}) 
    = \sum_j \Bigl(A_j^2 - i\frac{\partial A_j}{\partial x_j}\Bigr) + iA_nH + q - k^2.
\end{align*}
It follows from \eqref{3.12.1}, \eqref{3.25} and $c_0 = -V$ that
\begin{align}\label{3.28}
    E(x_{0}) 
    = \sum_\alpha \Bigl(A_\alpha^2 - i\frac{\partial A_\alpha}{\partial x_\alpha}\Bigr) + q - k^2.
\end{align}
From \eqref{3.13} we have
\begin{align*}
    (w_0 - iA_n)^2(A,q)
    = w_1^{-2}\biggl(\sum_\alpha A_\alpha\xi_\alpha\biggr)^2 +
    \biggl(w_1^{-2}\sum_\alpha \kappa_\alpha\xi_\alpha^2 - H\biggr)
    w_1^{-1}\sum_\alpha A_\alpha\xi_\alpha.
\end{align*}

For the two terms on the right-hand side in the equality above, we find that
\begin{align*}
    \int_{\mathbb{R}^{n-1}} e^{-w_1}
    \biggl(w_1^{-2}\sum_\alpha \kappa_\alpha\xi_\alpha^2 - H\biggr)
    w_1^{-1}\sum_\alpha A_\alpha\xi_\alpha \,d\xi^{\prime} = 0,
\end{align*}
and
\begin{align*}
    \int_{\mathbb{R}^{n-1}} e^{-w_1}w_1^{-2} \biggl(\sum_\alpha A_\alpha\xi_\alpha\biggr)^2 \,d\xi^{\prime}
    = \int_{\mathbb{R}^{n-1}} e^{-w_1}w_1^{-2}\sum_\alpha A_\alpha^2\xi_\alpha^2 \,d\xi^{\prime}.
\end{align*}
Actually, any term which is odd in $\xi_\alpha$ for any particular $\alpha$ will integrate to zero when integrating over $\mathbb{R}^{n-1}$. Thus, we only consider the terms which are even in $\xi_\alpha$ in an expression $f(\xi^{\prime})$ when integrating over $\mathbb{R}^{n-1}$. We henceforth denote by
\begin{equation*}
    f(\xi^{\prime}) \sim g(\xi^{\prime}),
\end{equation*}
where $g(\xi^{\prime})$ is the sum of the terms which are even in $\xi_\alpha$ in $f(\xi^{\prime})$. More precisely, $f(\xi^{\prime}) \sim g(\xi^{\prime})$ if and only if
\begin{equation*}
    \int_{\mathbb{R}^{n-1}} \big[f(\xi^{\prime}) - g(\xi^{\prime})\big] e^{-w_1} \,d\xi^{\prime} = 0.
\end{equation*}

Thus, in our notation, we have
\begin{equation}\label{3.28.1}
    (w_0 - iA_n)^2(A,q) \sim w_1^{-2}\sum_\alpha A_\alpha^2\xi_\alpha^2.
\end{equation}
Likewise,
\begin{align*}
    \frac{\partial (w_0 - iA_n)}{\partial x_n}(A,q) \sim 0, \qquad
    \sum_\alpha\frac{\partial w_1}{\partial \xi_\alpha}\frac{\partial A_{n}}{\partial x_{\alpha}} \sim 0.
\end{align*}
It is easy to compute that
\begin{align*}
    \frac{\partial (w_0 - iA_n)}{\partial x_\alpha}(A,q) 
    = w_1^{-1}\sum_\beta \frac{\partial A_\beta}{\partial x_\alpha}\xi_\beta,
\end{align*}
we then have
\begin{align*}
    w_{-1}(A,q)
    & \sim -\frac{1}{2}w_1^{-3}\sum_\alpha
    \Bigl(A_\alpha^2 - i\frac{\partial A_\alpha}{\partial x_\alpha}\Bigr)
    \xi_\alpha^2 + \frac{1}{2}w_1^{-1}
    \biggl[
        \sum_\alpha 
        \Bigl(A_\alpha^2 - i\frac{\partial A_\alpha}{\partial x_\alpha}\Bigr) 
        + q - k^2
    \biggr].
\end{align*}
Therefore, from \eqref{3.19} we get
\begin{align}\label{3.29}
    s_{-3}(A,q)
    & \sim \frac{1}{2}(2s_{-1}^3w_1^{-2} + s_{-1}^2w_1^{-3}) \sum_\alpha
    \Bigl(A_\alpha^2 - i\frac{\partial A_\alpha}{\partial x_\alpha}\Bigr)
    \xi_\alpha^2 \\
    & \quad - \frac{1}{2}s_{-1}^2w_1^{-1}\biggl[\sum_\alpha \Bigl(A_\alpha^2 - i\frac{\partial A_\alpha}{\partial x_\alpha}\Bigr) + q - k^2\biggr]. \notag
\end{align}
It follows from \eqref{3.9}, \eqref{3.29} and Lemma \ref{lem3.1} that
\begin{equation*}
    a_2(A,q) = -\frac{\Gamma(n-2)\operatorname{vol}(\mathbb{S}^{n-2})}{2(2\pi)^{n-1}}(q-k^2).
\end{equation*}
As a consequence, substituting this and \eqref{3.30} into \eqref{3.10}, we obtain
\begin{align*}
    a_2(x) 
    & = \frac{\Gamma(n-2)\operatorname{vol}(\mathbb{S}^{n-2})}{24(2\pi)^{n-1}(n^2-1)}
    \biggl[
        3(n^3-4n^2+n+8)H^2 + 3n(n-3) \sum_\alpha \kappa_\alpha^2 \\
        & \quad + 3(n+1)(n-2)\tilde{R} - (n+1)(n-4)R - 12(n^2-1)q + 12(n^2-1)k^2
    \biggr], \quad n \geqslant 3.
\end{align*}

\addvspace{2mm}

Step 4. According to \eqref{3.12} and \eqref{3.17}, we deduce that
\begin{align*}
    & \frac{\partial^2 s_{-2}}{\partial x_{\alpha}\partial x_{\beta}}(A,q)
    = -s_{-1}^2 \frac{\partial^2 m_0}{\partial x_{\alpha}\partial x_{\beta}}, \\
    & \frac{\partial^2 (w_0 - iA_n)}{\partial x_{\alpha}\partial x_{\beta}}(A,q)
    = -w_1^{-2} \frac{\partial^2 w_1}{\partial x_{\alpha}\partial x_{\beta}} \sum_\gamma A_\gamma \xi_\gamma + w_1^{-1} \sum_\gamma \frac{\partial^2 A_\gamma}{\partial x_{\alpha}\partial x_{\beta}} \xi_\gamma, \\
    & \frac{\partial^2 w_0}{\partial \xi_{\alpha}\partial \xi_{\beta}}(A,q)
    = 3w_1^{-5}\xi_{\alpha}\xi_{\beta} \sum_\gamma A_\gamma \xi_\gamma - w_1^{-3}
    \biggl(
        \delta_{\alpha\beta} \sum_\gamma A_\gamma \xi_\gamma + A_\alpha\xi_\beta + A_\beta\xi_\alpha
    \biggr),
\end{align*}
and
\begin{equation*}
    \frac{\partial^2 w_1}{\partial \xi_{\alpha}\partial \xi_{\beta}} 
    = -w_1^{-3}\xi_{\alpha}\xi_{\beta} + w_1^{-1}\delta_{\alpha\beta}.
\end{equation*}
Thus, we immediately get
\begin{align*}
    & \frac{\partial^2 (w_0 - iA_n)}{\partial x_{\alpha}\partial x_{\beta}}(A,q) \sim 0, &
    & \frac{\partial^2 w_0}{\partial \xi_{\alpha}\partial \xi_{\beta}}(A,q) \sim 0, \\
    & \biggl(
        \sum_{\alpha,\beta} \frac{\partial^2 w_1}{\partial \xi_{\alpha}\partial \xi_{\beta}}\frac{\partial^2 s_{-2}}{\partial x_{\alpha}\partial x_{\beta}}
    \biggr)(A,q) \sim 0, &
    & \biggl(
        \sum_{\alpha,\beta} \frac{\partial^2 w_0}{\partial \xi_{\alpha}\partial \xi_{\beta}}\frac{\partial^2 s_{-1}}{\partial x_{\alpha}\partial x_{\beta}}
    \biggr)(A,q) \sim 0.
\end{align*}
Consequently, it follows from \eqref{2.18} that
\begin{align*}
    s_{-4}(A,q) & \sim -s_{-1}
    \bigg[
        (w_0 - iA_n)s_{-3}+w_{-1}s_{-2}+w_{-2}s_{-1} - i\sum_\alpha
        \Bigl(
            \frac{\partial w_1}{\partial \xi_\alpha}\frac{\partial s_{-3}}{\partial x_\alpha} + \frac{\partial w_0}{\partial \xi_\alpha}\frac{\partial s_{-2}}{\partial x_\alpha}
        \Bigr)
    \bigg].
\end{align*}
By applying formula \eqref{2.16}, we find that
\begin{equation*}
    w_{-1}s_{-2}(x_{0}) = -s_{-1}^2(w_0 - iA_n)w_{-1}.
\end{equation*}
From \eqref{3.18}, we obtain
\begin{align*}
    \big[(w_0 - iA_n)s_{-3}\big](A,q)
    & \sim s_{-1}^3(w_0 - iA_n)^3 - s_{-1}^2(w_0 - iA_n)w_{-1} \\
    & \quad  - is_{-1}^3(w_0 - iA_n) \sum_\alpha \frac{\partial w_1}{\partial \xi_\alpha}\frac{\partial (w_0 - iA_n)}{\partial x_\alpha}.
\end{align*}
It follows from \eqref{2.9} that
\begin{align*}
    w_{-2}(A,q)
    & \sim \frac{1}{2}w_1^{-1}
    \biggl[
        -w_{-1} \big(H + 2(w_0 - iA_n)\big) + \frac{\partial w_{-1}}{\partial x_n} + i\sum_\alpha
        \Big(
            \frac{\partial w_1}{\partial \xi_\alpha}\frac{\partial w_{-1}}{\partial x_\alpha} + \frac{\partial w_0}{\partial \xi_\alpha}\frac{\partial (w_0 - iA_n)}{\partial x_\alpha} \\
            & \quad + i\frac{\partial w_0}{\partial \xi_\alpha}\frac{\partial A_{n}}{\partial x_{\alpha}}
        \Big)
        + \frac{i}{2} \sum_{\alpha,\beta} \frac{\partial^2 w_1}{\partial \xi_{\alpha}\partial \xi_{\beta}} \frac{\partial^2A_n}{\partial x_{\alpha}\partial x_{\beta}}
    \biggr]. \notag
\end{align*}
According to \eqref{3.18}, we have
\begin{align}\label{3.41}
    \frac{\partial s_{-3}}{\partial x_\alpha}(A,q)
    & = 2s_{-1}^3(w_0 - iA_n) \frac{\partial (w_0 - iA_n)}{\partial x_\alpha} + 3is_{-1}^4(w_0 - iA_n) \sum_\beta \frac{\partial w_1}{\partial \xi_\beta} \frac{\partial^2 w_1}{\partial x_{\alpha}\partial x_{\beta}} \\
    & \quad - s_{-1}^2 \frac{\partial w_{-1}}{\partial x_\alpha} - is_{-1}^3 \sum_\beta \frac{\partial w_1}{\partial \xi_\beta} \frac{\partial^2 (w_0 - iA_n)}{\partial x_{\alpha}\partial x_{\beta}} - is_{-1}^3 \sum_\beta \frac{\partial w_0}{\partial \xi_\beta} \frac{\partial^2 w_1}{\partial x_{\alpha}\partial x_{\beta}}. \notag
\end{align}
Note also that
\begin{align*}
    & \bigg(
        (w_0 - iA_n) \sum_{\alpha,\beta} \frac{\partial w_1}{\partial \xi_\alpha} \frac{\partial w_1}{\partial \xi_\beta} \frac{\partial^2 w_1}{\partial x_{\alpha}\partial x_{\beta}}
    \bigg)(A,q) \sim 0, \\
    & \bigg(
        \sum_{\alpha,\beta} \frac{\partial w_1}{\partial \xi_\alpha} \frac{\partial w_1}{\partial \xi_\beta} \frac{\partial^2 (w_0 - iA_n)}{\partial x_{\alpha}\partial x_{\beta}}
    \bigg)(A,q) \sim 0,
\end{align*}
and
\begin{align*}
    \bigg(
        \sum_{\alpha,\beta} \frac{\partial w_1}{\partial \xi_\alpha} \frac{\partial w_0}{\partial \xi_\beta} \frac{\partial^2 w_1}{\partial x_{\alpha}\partial x_{\beta}}
    \bigg)(A,q) \sim 0.
\end{align*}
It follows from \eqref{3.41} that
\begin{align*}
    \bigg(
        \sum_\alpha \frac{\partial w_1}{\partial \xi_\alpha} \frac{\partial s_{-3}}{\partial x_\alpha}
    \bigg)(A,q)
    & \sim 2s_{-1}^3(w_0 - iA_n) \sum_\alpha \frac{\partial w_1}{\partial \xi_\alpha} \frac{\partial (w_0 - iA_n)}{\partial x_\alpha} - s_{-1}^2 \sum_\alpha \frac{\partial w_1}{\partial \xi_\alpha} \frac{\partial w_{-1}}{\partial x_\alpha}.
\end{align*}
Using formula \eqref{3.17}, we get
\begin{align*}
    \bigg(
        \sum_\alpha \frac{\partial w_0}{\partial \xi_\alpha} \frac{\partial s_{-2}}{\partial x_\alpha}
    \bigg)(A,q)
    \sim -s_{-1}^2 \sum_\alpha \frac{\partial w_0}{\partial \xi_\beta} \frac{\partial (w_0 - iA_n)}{\partial x_\alpha}.
\end{align*}
We then conclude that
\begin{align*}
    s_{-4}(A,q) 
    & \sim -s_{-1}^4
    \biggl[
        (w_0 - iA_n)^3 - 3i(w_0 - iA_n) \sum_\alpha \frac{\partial w_1}{\partial \xi_\alpha}\frac{\partial (w_0 - iA_n)}{\partial x_\alpha}
    \biggr] \\
    & \quad - \frac{1}{2} s_{-1}^2w_1^{-1}
    \biggl(
            \frac{\partial w_{-1}}{\partial x_n} - Hw_{-1} - \sum_\alpha \frac{\partial w_0}{\partial \xi_\alpha} \frac{\partial A_{n}}{\partial x_{\alpha}} + \frac{i}{2} \sum_{\alpha,\beta} \frac{\partial^2 w_1}{\partial \xi_{\alpha}\partial \xi_{\beta}} \frac{\partial^2A_n}{\partial x_{\alpha}\partial x_{\beta}}
    \biggr) \\
    & \quad - \frac{1}{2} (2s_{-1}^3 + s_{-1}^2w_1^{-1})
    \biggl[
        i\sum_\alpha \frac{\partial w_1}{\partial \xi_\alpha} \frac{\partial w_{-1}}{\partial x_\alpha} + i\sum_\alpha \frac{\partial w_0}{\partial \xi_\alpha} \frac{\partial (w_0 - iA_n)}{\partial x_\alpha} - 2(w_0 - iA_n)w_{-1}
    \biggr].
\end{align*}
From \eqref{2.8}, it is straightforward to verify that
\begin{align*}
    \big[(w_0 - iA_n)w_{-1}\big](A,q)
    & \sim \frac{1}{2}w_1^{-1}
    \biggl[
        -(w_0 - iA_n)^3 + i(w_0 - iA_n) \sum_\alpha \frac{\partial w_1}{\partial \xi_\alpha} \frac{\partial (w_0 - iA_n)}{\partial x_\alpha} \\
        & \quad - H(w_0 - iA_n)^2 + (w_0 - iA_n)\frac{\partial (w_0 - iA_n)}{\partial x_n} \\
        & \quad - (w_0 - iA_n) \sum_\alpha \frac{\partial w_1}{\partial \xi_\alpha}\frac{\partial A_{n}}{\partial x_{\alpha}} + E(w_0 - iA_n)
    \biggr],
\end{align*}
\begin{align*}
    \frac{\partial w_{-1}}{\partial x_\alpha}(A,q)
    & \sim \frac{1}{2}w_1^{-1}
    \biggl[
        -2(w_0 - iA_n) \frac{\partial (w_0 - iA_n)}{\partial x_\alpha} + i \sum_\beta
        \bigg(
            \frac{\partial w_1}{\partial \xi_\beta} \frac{\partial^2 (w_0 - iA_n)}{\partial x_{\alpha}\partial x_{\beta}} \\
            & \quad + \frac{\partial w_0}{\partial \xi_\beta} \frac{\partial^2 w_1}{\partial x_{\alpha}\partial x_{\beta}}
        \bigg)
        - (w_0 - iA_n) \frac{\partial H}{\partial x_\alpha} - H \frac{\partial (w_0 - iA_n)}{\partial x_\alpha} + \frac{\partial^2 (w_0 - iA_n)}{\partial x_{n}\partial x_{\alpha}} \\
        & \quad - \sum_\beta \frac{\partial w_1}{\partial \xi_\beta} \frac{\partial^2 A_{n}}{\partial x_{\alpha}\partial x_{\beta}} + \frac{\partial E}{\partial x_\alpha}
    \biggr],
\end{align*}
and
\begin{align*}
    \frac{\partial w_{-1}}{\partial x_n}(A,q)
    & \sim -\frac{1}{2} w_1^{-2} \frac{\partial w_1}{\partial x_n}
    \biggl[
        -(w_0 - iA_n)^2 + i \sum_\alpha \frac{\partial w_1}{\partial \xi_\alpha} \frac{\partial (w_0 - iA_n)}{\partial x_\alpha} + E
    \biggr] \\
    & \quad + \frac{1}{2} w_1^{-1}
    \biggl[
        -2(w_0 - iA_n) \frac{\partial (w_0 - iA_n)}{\partial x_n} + i \sum_\alpha
        \biggl(
            \frac{\partial^2 w_1}{\partial \xi_{\alpha}\partial x_{n}} \frac{\partial (w_0 - iA_n)}{\partial x_\alpha} \\
            & \quad + \frac{\partial w_1}{\partial \xi_\alpha} \frac{\partial^2 (w_0 - iA_n)}{\partial x_{n}\partial x_{\alpha}} + \frac{\partial w_0}{\partial \xi_\alpha} \frac{\partial^2 w_1}{\partial x_{n}\partial x_{\alpha}}
        \biggr)
        + \frac{\partial E}{\partial x_n}
    \biggr].
\end{align*}
Finally, we obtain
\begin{align*}
    & s_{-4}(A,q)
    \sim -\frac{1}{2} (2s_{-1}^4 + 2s_{-1}^3w_1^{-1} + s_{-1}^2w_1^{-2}) (w_0 - iA_n)^3 + i(3s_{-1}^4 + 2s_{-1}^3w_1^{-1} + s_{-1}^2w_1^{-2})(w_0 - iA_n) \\
    & \quad \times \sum_\alpha \frac{\partial w_1}{\partial \xi_\alpha} \frac{\partial (w_0 - iA_n)}{\partial x_\alpha} - \frac{1}{4}
    \biggl[
        H(4s_{-1}^3w_1^{-1} + 3s_{-1}^2w_1^{-2}) + s_{-1}^2w_1^{-3} \frac{\partial w_1}{\partial x_n}
    \biggr]
    (w_0 - iA_n)^2 \\
    & \quad + (s_{-1}^3w_1^{-1} + s_{-1}^2w_1^{-2})
    \biggl[
        (w_0 - iA_n) \frac{\partial (w_0 - iA_n)}{\partial x_n} - \frac{i}{2} \sum_\alpha \frac{\partial w_1}{\partial \xi_\alpha} \frac{\partial^2 (w_0 - iA_n)}{\partial x_{n}\partial x_{\alpha}}
    \biggr] \\
    & \quad + \frac{1}{4} (2s_{-1}^3w_1^{-1} + s_{-1}^2w_1^{-2})
    \biggl[
        2E(w_0 - iA_n) + i(w_0 - iA_n) \sum_\alpha \frac{\partial w_1}{\partial \xi_\alpha} \frac{\partial (H + 2iA_n)}{\partial x_{\alpha}} \\
        & \quad + i \sum_{\alpha,\beta} \frac{\partial w_1}{\partial \xi_\alpha} \frac{\partial w_1}{\partial \xi_\beta} \frac{\partial^2 A_{n}}{\partial x_{\alpha}\partial x_{\beta}}
    \biggr]
    + \frac{i}{4}
    \biggl[
        2H(s_{-1}^3w_1^{-1}  + s_{-1}^2w_1^{-2}) + s_{-1}^2w_1^{-3} \frac{\partial w_1}{\partial x_n}
    \biggr] 
    \sum_\alpha \frac{\partial w_1}{\partial \xi_\alpha} \frac{\partial (w_0 - iA_n)}{\partial x_{\alpha}} \\
    & \quad + \frac{E}{4} \Bigl(Hs_{-1}^2w_1^{-2} + s_{-1}^2w_1^{-3} \frac{\partial w_1}{\partial x_n}\Bigr) + \frac{1}{2} s_{-1}^2w_1^{-1}
    \biggl(
        \sum_\alpha \frac{\partial w_0}{\partial \xi_\alpha} \frac{\partial A_{n}}{\partial x_{\alpha}} - \frac{i}{2} \sum_{\alpha,\beta} \frac{\partial^2 w_1}{\partial \xi_{\alpha}\partial \xi_{\beta}}\frac{\partial^2 A_{n}}{\partial x_{\alpha}\partial x_{\beta}}
    \biggr) \\
    & \quad - \frac{1}{4} s_{-1}^2w_1^{-2}
    \biggl[
        \frac{\partial E}{\partial x_n} + i \sum_\alpha
        \biggl(
            \frac{\partial^2 w_1}{\partial \xi_{\alpha}\partial x_{n}} \frac{\partial (w_0 - iA_n)}{\partial x_{\alpha}} + \frac{\partial w_0}{\partial \xi_{\alpha}} \frac{\partial^2 w_1}{\partial x_{n}\partial x_{\alpha}}
        \biggr)
    \biggr] \\
    & \quad - \frac{i}{2} (2s_{-1}^3 + s_{-1}^2w_1^{-1}) \sum_\alpha \frac{\partial w_0}{\partial \xi_\alpha} \frac{\partial (w_0 - iA_n)}{\partial x_{\alpha}} \\
    & := \sum_{j=1}^{10} s_{-4}^{(j)}.
\end{align*}

\addvspace{2mm}

Step 5. In what follows, we will compute the integrals of $s_{-4}^{(j)}\ (1 \leqslant j\leqslant 10)$ one by one. To begin with, we apply \eqref{3.13} to the first term $s_{-4}^{(1)}$. Thus,
\begin{equation*}
    (w_0 - iA_n)^3(A,q) \sim -\frac{3}{2} H w_1^{-2} \sum_\alpha A_\alpha^2\xi_{\alpha}^2 + \frac{3}{2} w_1^{-4} \sum_{\alpha,\beta} A_\alpha^2 \kappa_\beta \xi_{\alpha}^2 \xi_{\beta}^2,
\end{equation*}
which implies
\begin{align*}
    \frac{i}{(2 \pi)^{n}} \int_{\mathbb{R}^{n-1}}\int_{\mathcal{C}} s_{-4}^{(1)} e^{-\tau} \,d\tau \,d\xi^{\prime} 
    = \frac{n(n-2)\Gamma(n-3)\operatorname{vol}(\mathbb{S}^{n-2})}{4(2\pi)^{n-1}(n^2-1)} \sum_\alpha A_\alpha^2(nH - 2\kappa_\alpha).
\end{align*}
Then, we consider the term $(w_0 - iA_n) \sum_\alpha \frac{\partial w_1}{\partial \xi_\alpha} \frac{\partial (w_0 - iA_n)}{\partial x_\alpha}$ in $s_{-4}^{(2)}$. Since
\begin{align*}
    \biggl(
        (w_0 - iA_n) \sum_\alpha \frac{\partial w_1}{\partial \xi_\alpha} \frac{\partial (w_0 - iA_n)}{\partial x_\alpha}
    \biggr)(A,q)
    & \sim \frac{1}{2} w_1^{-4} \sum_{\alpha,\beta} \frac{\partial (A_\alpha \kappa_\beta)}{\partial x_{\alpha}} \xi_{\alpha}^2\xi_{\beta}^2 - \frac{1}{2} w_1^{-2} \sum_\alpha \frac{\partial (A_\alpha H)}{\partial x_{\alpha}} \xi_{\alpha}^2,
\end{align*}
we obtain
\begin{align*}
    \frac{i}{(2 \pi)^{n}} \int_{\mathbb{R}^{n-1}}\int_{\mathcal{C}} s_{-4}^{(2)} e^{-\tau} \,d\tau \,d\xi^{\prime} 
    = \frac{i(n-2)\Gamma(n-3)\operatorname{vol}(\mathbb{S}^{n-2})}{4(2\pi)^{n-1}(n+1)} \sum_\alpha \frac{\partial [A_\alpha (2\kappa_\alpha - nH)]}{\partial x_{\alpha}}.
\end{align*}
It is easy to calculate that
\begin{align*}
    \frac{\partial w_1}{\partial x_n} = w_1^{-1} \sum_\alpha \kappa_\alpha \xi_\alpha^2,
\end{align*}
we then have
\begin{align*}
    & \frac{i}{(2 \pi)^{n}} \int_{\mathbb{R}^{n-1}}\int_{\mathcal{C}} s_{-4}^{(3)} e^{-\tau} d\tau \,d\xi^{\prime} 
    = -\frac{\Gamma(n-3)\operatorname{vol}(\mathbb{S}^{n-2})}{4(2\pi)^{n-1}(n^2-1)} \sum_\alpha A_\alpha^2 \big[2\kappa_\alpha + (2n^2 - n - 2)H\big].
\end{align*}
Considering the terms $(w_0 - iA_n) \frac{\partial (w_0 - iA_n)}{\partial x_n}$ and $\sum_\alpha \frac{\partial w_1}{\partial \xi_\alpha} \frac{\partial^2 (w_0 - iA_n)}{\partial x_{n}\partial x_{\alpha}}$ in $s_{-4}^{(4)}$, we find that
\begin{align*}
    \biggl((w_0 - iA_n) \frac{\partial (w_0 - iA_n)}{\partial x_n}\biggr)(A,q)
    & \sim \frac{1}{2} w_1^{-2} \sum_\alpha
    \Bigl(
        iA_\alpha \frac{\partial (H - 2\kappa_\alpha)}{\partial x_{\alpha}} + \frac{\partial A_\alpha^2}{\partial x_{n}}
    \Bigr)
    \xi_{\alpha}^2 \\
    & \quad + \frac{1}{2} w_1^{-4} \sum_{\alpha,\beta}
    \Bigl(
        iA_\alpha \frac{\partial \kappa_\beta}{\partial x_{\alpha}} - 2 A_\alpha^2\kappa_\beta
    \Bigr)
    \xi_{\alpha}^2 \xi_{\beta}^2,
\end{align*}
and
\begin{align*}
    \biggl(
        \sum_\alpha \frac{\partial w_1}{\partial \xi_\alpha} \frac{\partial^2 (w_0 - iA_n)}{\partial x_{n}\partial x_{\alpha}}
    \biggr)(A,q)
    & \sim -w_1^{-4} \sum_{\alpha,\beta} \frac{\partial (A_\alpha\kappa_\beta)}{\partial x_{\alpha}} \xi_{\alpha}^2 \xi_{\beta}^2 + w_1^{-2} \sum_\alpha \frac{\partial^2 A_\alpha}{\partial x_{n}\partial x_{\alpha}} \xi_{\alpha}^2.
\end{align*}
It follows from this and \eqref{3.28.1} that
\begin{align*}
    \frac{i}{(2 \pi)^{n}} \int_{\mathbb{R}^{n-1}}\int_{\mathcal{C}} s_{-4}^{(4)} e^{-\tau} \,d\tau \,d\xi^{\prime}
    & = \frac{\Gamma(n-3)\operatorname{vol}(\mathbb{S}^{n-2})}{4(2\pi)^{n-1}(n+1)} \sum_\alpha
    \biggl[
        iA_\alpha \frac{\partial }{\partial x_{\alpha}} \big((n+3)H - 2(n-1)\kappa_\alpha\big) \\
        & \quad - (H + 2\kappa_\alpha)
        \Bigl( 2A_\alpha^2 - i \frac{\partial A_\alpha}{\partial x_{\alpha}} \Bigr)
        + (n + 1) \frac{\partial }{\partial x_{n}}
        \Bigl( A_\alpha^2 - i \frac{\partial A_\alpha}{\partial x_{\alpha}} \Bigr)
    \biggr].
\end{align*}
Next, for the fifth term $s_{-4}^{(5)}$, we see that
\begin{align*}
    \big[E(w_0 - iA_n)\big](A,q) \sim \frac{E}{2}
    \biggl( w_1^{-2} \sum_\alpha \kappa_\alpha \xi_\alpha^2 - H \biggr),
\end{align*}
\begin{align*}
    \biggl(
        (w_0 - iA_n) \sum_\alpha \frac{\partial w_1}{\partial \xi_\alpha} \frac{\partial (H + 2iA_n)}{\partial x_{\alpha}}
    \biggr)(A,q)
    \sim w_1^{-2} \sum_\alpha A_\alpha \frac{\partial (H + 2iA_n)}{\partial x_{\alpha}} \xi_\alpha^2,
\end{align*}
and
\begin{align*}
    \biggl(
        \sum_{\alpha,\beta} \frac{\partial w_1}{\partial \xi_\alpha} \frac{\partial w_1}{\partial \xi_\beta} \frac{\partial^2 A_{n}}{\partial x_{\alpha}\partial x_{\beta}}
    \biggr)(A,q)
    \sim w_1^{-2} \sum_\alpha \frac{\partial^2 A_{n}}{\partial x_{\alpha}^2} \xi_\alpha^2.
\end{align*}
As a result, we get
\begin{align*}
    & \frac{i}{(2 \pi)^{n}} \int_{\mathbb{R}^{n-1}}\int_{\mathcal{C}} s_{-4}^{(5)} e^{-\tau} \,d\tau \,d\xi^{\prime} \\
    & \quad = -\frac{(n-2)\Gamma(n-3)\operatorname{vol}(\mathbb{S}^{n-2})}{4(2\pi)^{n-1}(n-1)}
    \biggl[
        (n-2)HE - i \sum_\alpha
        \biggl(
            A_\alpha \frac{\partial (H + 2iA_n)}{\partial x_{\alpha}} + \frac{\partial^2 A_{n}}{\partial x_{\alpha}^2}
        \biggr)
    \biggr].
\end{align*}
Noting that
\begin{align*}
    \biggl(
        \sum_\alpha \frac{\partial w_1}{\partial \xi_\alpha} \frac{\partial (w_0 - iA_n)}{\partial x_{\alpha}}
    \biggr)(A,q)
    \sim w_1^{-2} \sum_\alpha \frac{\partial A_{\alpha}}{\partial x_{\alpha}} \xi_\alpha^2,
\end{align*}
we obtain
\begin{align*}
    \frac{i}{(2 \pi)^{n}} \int_{\mathbb{R}^{n-1}}\int_{\mathcal{C}} s_{-4}^{(6)} e^{-\tau} \,d\tau \,d\xi^{\prime} 
    = \frac{i\Gamma(n-3)\operatorname{vol}(\mathbb{S}^{n-2})}{4(2\pi)^{n-1}(n^2-1)} \sum_\alpha \frac{\partial A_{\alpha}}{\partial x_{\alpha}} (n^2H + 2 \kappa_\alpha).
\end{align*}
For the seventh term $s_{-4}^{(7)}$, a simple computation shows that 
\begin{align*}
    \frac{i}{(2 \pi)^{n}} \int_{\mathbb{R}^{n-1}}\int_{\mathcal{C}} s_{-4}^{(7)} e^{-\tau} \,d\tau \,d\xi^{\prime} 
    = \frac{n\Gamma(n-3)\operatorname{vol}(\mathbb{S}^{n-2})}{4(2\pi)^{n-1}(n-1)} HE.
\end{align*}
We consider the terms $\sum_\alpha \frac{\partial w_0}{\partial \xi_\alpha} \frac{\partial A_{n}}{\partial x_{\alpha}}$ and $\sum_{\alpha,\beta} \frac{\partial^2 w_1}{\partial \xi_{\alpha}\partial \xi_{\beta}}\frac{\partial^2 A_{n}}{\partial x_{\alpha}\partial x_{\beta}}$ in $s_{-4}^{(8)}$,
\begin{align*}
    \biggl(
        \sum_\alpha \frac{\partial w_0}{\partial \xi_\alpha} \frac{\partial A_{n}}{\partial x_{\alpha}}
    \biggr)(A,q)
    \sim -w_1^{-3} \sum_\alpha A_\alpha \frac{\partial A_{n}}{\partial x_{\alpha}} \xi_\alpha^2 + w_1^{-1} \sum_\alpha A_\alpha \frac{\partial A_{n}}{\partial x_{\alpha}},
\end{align*}
and
\begin{align*}
    \biggl(
        \sum_{\alpha,\beta} \frac{\partial^2 w_1}{\partial \xi_{\alpha}\partial \xi_{\beta}}\frac{\partial^2 A_{n}}{\partial x_{\alpha}\partial x_{\beta}}
    \biggr)(A,q)
    \sim -w_1^{-3} \sum_\alpha \frac{\partial^2 A_{n}}{\partial x_{\alpha}^2} \xi_\alpha^2 + w_1^{-1} \sum_\alpha \frac{\partial^2 A_{n}}{\partial x_{\alpha}^2}.
\end{align*}
Accordingly,
\begin{align*}
    \frac{i}{(2 \pi)^{n}} \int_{\mathbb{R}^{n-1}}\int_{\mathcal{C}} s_{-4}^{(8)} e^{-\tau} \,d\tau \,d\xi^{\prime} 
    = \frac{(n-2)\Gamma(n-3)\operatorname{vol}(\mathbb{S}^{n-2})}{4(2\pi)^{n-1}(n-1)} \sum_\alpha
    \Bigl(
        2A_\alpha \frac{\partial A_{n}}{\partial x_{\alpha}} - i \frac{\partial^2 A_{n}}{\partial x_{\alpha}^2}
    \Bigr).
\end{align*}
For the ninth term $s_{-4}^{(9)}$, we have
\begin{align*}
    \biggl(
        \sum_\alpha \frac{\partial^2 w_1}{\partial \xi_{\alpha}\partial x_{n}} \frac{\partial (w_0 - iA_n)}{\partial x_{\alpha}}
    \biggr)(A,q)
    \sim -w_1^{-4} \sum_{\alpha,\beta} \frac{\partial A_{\alpha}}{\partial x_{\alpha}} \kappa_\beta \xi_\alpha^2 \xi_\beta^2 + 2w_1^{-2} \sum_\alpha \frac{\partial A_{\alpha}}{\partial x_{\alpha}} \kappa_\alpha \xi_\alpha^2,
\end{align*}
and
\begin{align*}
    \biggl(
        \sum_\alpha \frac{\partial w_0}{\partial \xi_{\alpha}} \frac{\partial^2 w_1}{\partial x_{n}\partial x_{\alpha}}
    \biggr)(A,q)
    & \sim -w_1^{-4} \sum_{\alpha,\beta} A_{\alpha} \frac{\partial \kappa_\beta}{\partial x_{\alpha}} \xi_\alpha^2 \xi_\beta^2 + w_1^{-2} \sum_\alpha A_{\alpha} \frac{\partial \kappa_\beta}{\partial x_{\alpha}} \xi_\beta^2.
\end{align*}
Consequently,
\begin{align*}
    & \frac{i}{(2 \pi)^{n}} \int_{\mathbb{R}^{n-1}}\int_{\mathcal{C}} s_{-4}^{(9)} e^{-\tau} \,d\tau \,d\xi^{\prime} \\
    & \quad = \frac{\Gamma(n-3)\operatorname{vol}(\mathbb{S}^{n-2})}{4(2\pi)^{n-1}(n^2-1)}
    \biggl[
        i \sum_\alpha
        \biggl(
            \frac{\partial A_{\alpha}}{\partial x_{\alpha}} (H - 2n\kappa_\alpha) - A_{\alpha} \frac{\partial (nH - 2\kappa_\alpha)}{\partial x_{\alpha}}
        \biggr)
        - (n^2-1) \frac{\partial E}{\partial x_{n}}
    \biggr].
\end{align*}
Considering $\sum_\alpha \frac{\partial w_0}{\partial \xi_\alpha} \frac{\partial (w_0 - iA_n)}{\partial x_{\alpha}}$ in the last term $s_{-4}^{(10)}$, we obtain
\begin{align*}
    & \biggl(
        \sum_\alpha \frac{\partial w_0}{\partial \xi_\alpha} \frac{\partial (w_0 - iA_n)}{\partial x_{\alpha}}
    \biggr)(A,q) \sim -\frac{1}{2} w_1^{-5} \sum_{\alpha,\beta}
    \Bigl(
        2 \frac{\partial A_{\alpha}}{\partial x_{\alpha}} \kappa_\beta + A_{\alpha} \frac{\partial \kappa_\beta}{\partial x_{\alpha}}
    \Bigr)
    \xi_\alpha^2 \xi_\beta^2 \\
    & \quad + \frac{1}{2} w_1^{-3} \sum_{\alpha,\beta} A_{\alpha} \frac{\partial \kappa_\beta}{\partial x_{\alpha}} (\xi_\alpha^2 + \xi_\beta^2) + w_1^{-3} \sum_{\alpha} \frac{\partial A_{\alpha}}{\partial x_{\alpha}} \kappa_\alpha \xi_\alpha^2 - \frac{1}{2} w_1^{-1} \sum_{\alpha,\beta} A_{\alpha} \frac{\partial \kappa_\beta}{\partial x_{\alpha}}.
\end{align*}
Thus, we get
\begin{align*}
    & \frac{i}{(2 \pi)^{n}} \int_{\mathbb{R}^{n-1}}\int_{\mathcal{C}} s_{-4}^{(10)} e^{-\tau} \,d\tau \,d\xi^{\prime} \\
    & \quad = -\frac{i(n-2)\Gamma(n-3)\operatorname{vol}(\mathbb{S}^{n-2})}{4(2\pi)^{n-1}(n^2-1)} \sum_\alpha
    \biggl[
        2 \frac{\partial A_{\alpha}}{\partial x_{\alpha}} \bigl((n-1)\kappa_\alpha - H\bigr) - A_{\alpha} \frac{\partial }{\partial x_{\alpha}} \bigl((n^2-2n-2)H + 2\kappa_\alpha\bigr)
    \biggr].
\end{align*}
Combining these ten integrals above, we conclude that
\begin{align}\label{3.38.1}
    a_3(A,q)
    & = \frac{\Gamma(n-3)\operatorname{vol}(\mathbb{S}^{n-2})}{4(2\pi)^{n-1}}
    \biggl[
        (n-4)H \sum_\alpha
        \Bigl( A_\alpha^2 - i \frac{\partial A_{\alpha}}{\partial x_{\alpha}} \Bigr) 
        + \sum_\alpha \frac{\partial }{\partial x_{n}}
        \Bigl( A_\alpha^2 - i \frac{\partial A_{\alpha}}{\partial x_{\alpha}} \Bigr) \\
        & \quad + i \sum_\alpha A_{\alpha} \frac{\partial H}{\partial x_{\alpha}} - 2 \sum_\alpha A_\alpha^2 \kappa_\alpha - (n-4)HE - \frac{\partial E}{\partial x_{n}}
    \biggr]. \notag
\end{align}

To complete the proof, we need to simplify the expression above. Recall that
\begin{equation*}
    b = \frac{1}{2}\sum_{\alpha,\beta} g^{\alpha\beta} \frac{\partial g_{\alpha\beta}}{\partial x_n} + 2iA_n.
\end{equation*}
In the boundary coordinates, at the point $x_{0}$, we find that
\begin{align*}
    \frac{\partial b}{\partial x_{n}}(x_{0}) 
    = \frac{1}{2} \sum_\alpha \frac{\partial^2 g_{\alpha\alpha}}{\partial x_n^2} - 2 \sum_\alpha \kappa_\alpha^2 + 2i \frac{\partial A_{n}}{\partial x_{n}}.
\end{align*}
It follows from \eqref{3.26} that
\begin{align*}
    \frac{\partial V}{\partial x_{n}}(x_{0})
    & = \frac{\partial }{\partial x_{n}}
    \biggl[
        \sum_{j}
        \Bigl(A_j^2 - i \frac{\partial A_j}{\partial x_j}\Bigr)
        + q
    \biggr]
    + 2iA_n \sum_\alpha \kappa_\alpha^2 + i \sum_\alpha A_\alpha \frac{\partial H}{\partial x_\alpha} + i \frac{\partial A_{n}}{\partial x_n} \sum_\alpha \kappa_\alpha \\
    & \quad - 2 \sum_\alpha A_\alpha^2 \kappa_\alpha - \frac{i}{2} A_{n} \sum_\alpha \frac{\partial^2 g_{\alpha\alpha}}{\partial x_n^2}.
\end{align*}
Combining this and \eqref{3.25}, we obtain
\begin{align*}
    \frac{\partial E}{\partial x_{n}}(x_{0})
    = \frac{\partial }{\partial x_{n}}
    \biggl[
        \sum_\alpha
        \Bigl( A_\alpha^2 - i \frac{\partial A_{\alpha}}{\partial x_{\alpha}} \Bigr)
        + q
    \biggr]
    + i \sum_\alpha A_{\alpha} \frac{\partial H}{\partial x_{\alpha}} - 2 \sum_\alpha A_{\alpha}^2 \kappa_\alpha.
\end{align*}
Substituting this and \eqref{3.28} into \eqref{3.38.1}, we then get
\begin{align}\label{3.39}
    a_3(A,q) 
    = -\frac{\Gamma(n-3)\operatorname{vol}(\mathbb{S}^{n-2})}{4(2\pi)^{n-1}}
    \biggl[\frac{\partial q}{\partial x_{n}} + (n-4)Hq - (n-4)k^2H\biggr].
\end{align}
Finally, we obtain $a_3(x)$ in Theorem \ref{thm1.1} by applying \eqref{3.10}, \eqref{3.31} and \eqref{3.39}.
\end{proof}

\addvspace{5mm}

\begin{remark}
    Using formulas \eqref{2.9.1} and \eqref{2.14.1}, one may further calculate the lower order symbols $w_{-1-m}\ (m \geqslant 2)$ and $s_{-1-m}\ (m \geqslant 4)$. Therefore, we can get $a_k(x)$ for $k \geqslant 4$ by formula \eqref{3.1}.
\end{remark}

\addvspace{5mm}

\begin{proof}[Proof of Corollary \ref{cor1.3}]
    Since the sectional curvature of $\Omega$ is a constant $K$, the Riemann curvature tensor has the form (see p.~183 of \cite{ChowKnopf04})
    \begin{align*}
        \tilde{R}_{ijkl} = K(g_{il}g_{jk} - g_{ik}g_{jl}).
    \end{align*}
    In this case, it follows from \eqref{ricci} and \eqref{scalar} that
    \begin{align*}
        \tilde{R}_{ij} = (n-1)K g_{ij},
    \end{align*}
    and
    \begin{align}\label{3.19.1}
        \tilde{R} = n(n-1)K.
    \end{align}
    In the boundary normal coordinates, at the point $x_{0}$, we have
    \begin{align}\label{3.22.3}
        \tilde{R}_{ij} = (n-1)K \delta_{ij}.
    \end{align}
    According to the Gauss equation (see p.~39--40 of \cite{ChowLL06})
    \begin{align*}
        R_{\alpha\beta\gamma\rho} 
        = \tilde{R}_{\alpha\beta\gamma\rho} + h_{\alpha\rho}h_{\beta\gamma} - h_{\alpha\gamma}h_{\beta\rho},
    \end{align*}
    one has
    \begin{align*}
        R_{\alpha\beta} 
        = \tilde{R}_{\alpha\beta} - \tilde{R}_{n \alpha\beta n} + Hh_{\alpha\beta} - \sum_{\gamma} h_{\alpha\gamma}h_{\gamma\beta},
    \end{align*}
    and
    \begin{align*}
        R = \tilde{R} - 2\tilde{R}_{nn} + H^2 - |h|^2.
    \end{align*}
    We recall that $h_{\alpha\beta}(x_{0}) = \kappa_{\alpha}\delta_{\alpha\beta}$, thus
    \begin{align}\label{3.22.1}
        R_{\alpha\alpha} = (n-2)K + H\kappa_{\alpha} - \kappa_{\alpha}^2,
    \end{align}
    and
    \begin{align}\label{3.22.2}
        \sum_{\alpha} \kappa_{\alpha}^2 = (n-1)(n-2)K + H^2 - R.
    \end{align}
    Substituting \eqref{3.19.1} and \eqref{3.22.2} into the expression for $a_2(x)$ in Theorem \ref{thm1.1}, we immediately get $a_2(x)$ in Corollary \ref{cor1.3}.

    Next, multiplying by $\kappa_{\alpha}$ on both sides of \eqref{3.22.1}, and using \eqref{3.22.2}, we obtain
    \begin{align}\label{3.24}
        \sum_{\alpha} \kappa_{\alpha} R_{\alpha\alpha} 
        = H^3 + n(n-2)HK - HR - \sum_{\alpha} \kappa_{\alpha}^3.
    \end{align}
    It follows from formula \eqref{covariant} that
    \begin{align*}
        \nabla_{\frac{\partial }{\partial x_n}}\tilde{R}_{nn} 
        = \frac{\partial \tilde{R}_{nn}}{\partial x_n} - 2\sum_i \Gamma^{i}_{nn} \tilde{R}_{in}.
    \end{align*}
    In boundary normal coordinates, from \eqref{chris} we have
    \begin{align*}
        \Gamma^{i}_{nn} = \frac{1}{2} \sum_j g^{ij}
        \Bigl(
            2\frac{\partial g_{nj}}{\partial x_n} - \frac{\partial g_{nn}}{\partial x_j}
        \Bigr)
        = 0.
    \end{align*}
    Combining this and \eqref{3.22.3}, we obtain $\nabla_{\frac{\partial }{\partial x_n}}\tilde{R}_{nn} = 0$. Substitute \eqref{3.19.1}, \eqref{3.22.2}, \eqref{3.22.3} and \eqref{3.24} into the expression for $a_3(x)$ in Theorem \ref{thm1.1}. Finally, we get $a_3(x)$ in Corollary \ref{cor1.3}.
\end{proof}

\addvspace{10mm}

\section*{Acknowledgements}
This research was supported by NNSF of China (11671033/A010802) and NNSF of China \\
(11171023/A010801).

\addvspace{10mm}

\end{document}